\documentclass{amsproc}
\usepackage{euscript}
\usepackage{cases}
\usepackage{mathrsfs}
\usepackage{bbm}
\usepackage{amssymb}
\usepackage{amsfonts,amsmath,amsxtra,mathdots,mathabx}
\usepackage{color}
\usepackage{hyperref}
\usepackage{tikz}
\usepackage{appendix,upgreek}

\allowdisplaybreaks

\DeclareFontFamily{U}{matha}{\hyphenchar\font45}
\DeclareFontShape{U}{matha}{m}{n}{
	<5> <6> <7> <8> <9> <10> gen * matha
	<10.95> matha10 <12> <14.4> <17.28> <20.74> <24.88> matha12
}{}
\DeclareSymbolFont{matha}{U}{matha}{m}{n}

\DeclareMathSymbol{\Lt}{3}{matha}{"CE}
\DeclareMathSymbol{\Gt}{3}{matha}{"CF}

\DeclareSymbolFont{mathc}{OML}{txmi}{m}{it}
\DeclareMathSymbol{\varuu}{\mathord}{mathc}{117}
\DeclareMathSymbol{\varvv}{\mathord}{mathc}{118}
\DeclareMathSymbol{\varww}{\mathord}{mathc}{119}

\def\valpha{\text{\scalebox{0.84}[1.02]{$\alpha$}}}   
\def\vepsilon{\upvarepsilon}
\def\vnu{\text{{\scalebox{0.9}[1]{$\nu$}}}}

\newcommand{\BR}{{\mathbb {R}}} 
\newcommand{\BZ}{{\mathbb {Z}}}

\newcommand{\GL}{{\mathrm {GL}}}

\newcommand{\SL}{{\mathrm {SL}}}

\newcommand{\ra}{\rightarrow} 
\def\sumx{\sideset{}{^\star}\sum}
\def\mod{\mathrm{mod}\,  }
\def\nd{\mathrm{d}}

\def\lp {\left (}
\def\rp {\right )}

\def\shskip{\hspace{0.5pt}}

\newcommand{\delete}[1]{}

\theoremstyle{plain}

\newtheorem{thm}{Theorem} \newtheorem{cor}[thm]{Corollary}
\newtheorem{lem}{Lemma}[section]

\theoremstyle{remark} 
\newtheorem{remark}{Remark}[section] 
\newtheorem{defn}{Definition}[section]

\newtheorem*{acknowledgement}{Acknowledgements}

\numberwithin{equation}{section}

\begin{document}
	
	\title[$\mathrm{GL}_3 \hspace{-1pt} \times \hspace{-1pt} \mathrm{GL}_2$\,$L$-functions at Special Points]{{The Second Moment of  $\mathrm{GL}_3 \times \mathrm{GL}_2$   $L$-functions at Special Points}}

	\begin{abstract}
	Let $\phi$ be a fixed Hecke--Maass form for  $\mathrm{SL}_3 (\mathbb{Z})$ and $u_j $ traverse an orthonormal basis of Hecke--Maass  forms for  $\mathrm{SL}_2 (\BZ) $. Let $1/4+t_j^2$ be the Laplace eigenvalue of $u_j $. In this paper, we prove the mean  Lindel\"of hypothesis  for the second moment of $ L (1/2+it_j, \phi \times u_j) $ on $ T < t_j \leqslant T + \sqrt{T} $. Previously, this was proven by Young on $ t_j \leqslant T$. Our approach is more direct as we do not  apply the Poisson summation formula to detect the `Eisenstein--Kloosterman' cancellation. 
	\end{abstract}
	
	\author{Zhi Qi}
	\address{School of Mathematical Sciences\\ Zhejiang University\\Hangzhou, 310027\\China}
	\email{zhi.qi@zju.edu.cn}
	
	\thanks{The author was supported by National Key R\&D Program of China No. 2022YFA1005300.}

	
	\maketitle

	\section{Introduction}
	
	
	Let $\{u_j (z)\}$ be an orthonormal basis of (even or odd) Hecke--Maass cusp forms on  $\mathrm{SL}_2 (\BZ) \backslash \mathbb{H}^2$.   Let $\lambda_j = s_j (1-s_j)$  
	be the Laplace eigenvalue of $u_j (z)$,  
	with $s_j = 1/2+ i t_j$ ($t_j > 0$). 
		The Fourier expansion of $u_j (z)$ reads:
	\begin{align*}
		u_j (x+iy) =   \sqrt{y} \sum_{n \neq 0}  \rho_j (n) K_{i t_j} (2\pi |n| y) e (n x), 
	\end{align*}
	where as usual $K_{\vnu} (x)$ is the $K$-Bessel function and $e (x) = \exp (2\pi i x)$. Let $\lambda_j (n)$  be the $n$-th Hecke eigenvalue of $ u_j (z) $. It is well known that $ \rho_j (  n) =   \lambda_j (n) \rho_j (  1)  $ and $ \lambda_j (n) $ is real-valued for any $n \geqslant 1$. Define the harmonic weight 
	\begin{align*}
		\omega_j = \frac {|\rho_j (1)|^2} {\cosh \pi t_j} .    
	\end{align*}

\delete{Let $\phi_2$ be a (holomorphic or Maass) newform for a Hecke congruence subgroup of  $ \SL_2 (\BZ) $, $\phi_3$ or $\phi_4$ be a Hecke--Maass form for $\SL_3 (\BZ)$ or $\SL_4 (\BZ)$ respectively. For simplicity, we shall write $\phi$ if there is no confusion in the context. }

For any fixed Hecke--Maass cusp form $\phi$ for $\SL_3 (\BZ)$ or $\SL_4 (\BZ)$,  Young  \cite{Young-GL(3)-Special-Points} and  Chandee--Li  \cite{Chandee-Li-GL(4)-Special-Points}  proved  that the second moment of the Rankin--Selberg $L$-function $L (s, \phi \times u_j)$ at the special point $s=s_j$ satisfies the mean Lindel\"of hypothesis:  
	\begin{align}\label{1eq: mean Lindelof, T}
		\sum_{t_j \leqslant T}   \left|L (s_j, \phi  \times u_j) \right|^2 \Lt_{\phi, \vepsilon}   T^{2 +\vepsilon}. 
	\end{align} 

The purpose of this paper is to extend Young's $\GL_3 \times \GL_2$ mean Lindel\"of bound  to the short-interval case for $ T < t_j \leqslant T + \sqrt{T} $.  

\begin{thm}\label{thm G:(3)}
	Let $ \phi $ be an  $\SL_3 (\BZ)$ Hecke--Maass cusp form. Then for $\sqrt{T} \leqslant M \leqslant T$ we have
	\begin{align}\label{1eq: mean Lindelof, GL(3)}
		\sum_{T < t_j \leqslant T + M}   \left|L (s_j, \phi  \times u_j) \right|^2 \Lt_{\phi, \vepsilon}   M T^{1 +\vepsilon}, 
	\end{align} 
where the implied constant depends only  on $\phi$ and $\vepsilon$. 
\end{thm}

\delete{The techniques underlying the proof of Theorem \ref{thm G:(3)} can be modified to yield
\begin{align}
	\sum_{T < t_j \leqslant T + M}   \left|L (s_j,      u_j) \right|^6 \Lt_{\phi, \vepsilon}   M T^{1 +\vepsilon} ,
\end{align} 
for $ \sqrt{T} \leqslant M \leqslant T $. See \cite[\S 10]{Young-GL(3)-Special-Points}. }

Deshouillers, Iwaniec, and Luo \cite{DI-Nonvanishing,Luo-Twisted-LS} established the following large sieve inequalities for the special twisted Hecke eigenvalues $ \lambda_{j} (n) n^{it_j}  $: 
\begin{align}\label{1eq: DI's bound}
	\sum_{t_j \leqslant T} \omega_j  \bigg| \sum_{ n \leqslant N}  a_{n} \lambda_{j} (n) n^{it_j} \bigg|^2 \Lt \big(T^2 +  N^{2}\big) (TN)^{\vepsilon} \sum_{ n \leqslant N}  |a_{n}|^2 ,  
\end{align}
\begin{align}\label{1eq: Luo's bound, 1}
	\sum_{t_j \leqslant T} \omega_j   \bigg| \sum_{ n \leqslant N}  a_{n} \lambda_{j} (n) n^{it_j} \bigg|^2 \hskip -2pt \Lt \hskip -1pt  \big(T^2 \hskip -1pt    +  \hskip -1pt    T^{3/2} N^{1/2} \hskip -1pt   + \hskip -1pt   N^{5/4} \big) (TN)^{\vepsilon} \hskip -2pt \sum_{ n \leqslant N} \hskip -2pt |a_{n}|^2,
\end{align}
for any complex $a_n$. 
Note that \eqref{1eq: Luo's bound, 1} improves \eqref{1eq: DI's bound} for $N > T$. 

The proof of \eqref{1eq: mean Lindelof, T} by Young  \cite{Young-GL(3)-Special-Points} (also that by Chandee--Li  \cite{Chandee-Li-GL(4)-Special-Points}) utilizes a refinement of Luo's large sieve \eqref{1eq: Luo's bound, 1} in asymptotic form,   which is conducive to further analysis using the   Vorono\"i summation for the Fourier coefficients of $\phi$. More precisely, Young adapted an idea from Iwaniec and Li \cite{Iwaniec-Li-Ortho}  and applied the Poisson summation formula instead of the Euler--Maclaurin formula by Luo \cite{Luo-Twisted-LS}.

In the spirit of Young, we establish an asymptotic large sieve in the next technical theorem so that his hybrid large sieve (Lemma \ref{lem: Young's LS}) yields Theorem \ref{thm G:(3)} after the application of the Vorono\"i summation formula for $\SL_3 (\BZ)$. 

For
\begin{equation}\label{1eq: weight k} 
	h ( t  ) = \exp \bigg( \hspace{-1pt}  -  \frac {(t - T)^2} {M^2}    \bigg) + \exp \bigg(  \hspace{-1pt} -  \frac {(t + T)^2} {M^2}    \bigg), 
\end{equation}  
define 
\begin{align}  \label{1eq: S, cusp}	
	S (\mathcal{A})   & =	   \sum_{j = 1 }^{\infty}   \omega_j { h  ( t_j ) }   \bigg| \sum_{ N < n \leqslant 2 N }  a_{n} \lambda_{j} (n) n^{i t_j}  \bigg|^2   ,  \\
	\label{1eq: E, Eis}		 T (\mathcal{A})  &  =	   \frac 1 {\pi} 	\int_{-\infty}^{\infty}   \frac{h (t)}    {|\zeta(1+2it)|^2} \bigg| \sum_{ N < n \leqslant 2 N }   {a}_{n} \sigma_{  2 i t} (n)   \bigg|^2  \nd t ,
\end{align} 
where  $\mathcal{A} $ stands for the sequence $\{a_n\}$ and $\sigma_{\vnu} (n)$ is the divisor function 
$$
	\sigma_{\vnu} (n) = \sum_{d | n} d^{\hspace{0.5pt} \vnu} .  
$$ 

\begin{thm}
	\label{thm: asymptotic large sieve}
	
	Let $ T^{\vepsilon} \leqslant M \leqslant T^{1-\vepsilon}$. Assume $\mathcal{A}$ is real-valued. Then  we have
	\begin{align}
		S (\mathcal{A})  + T (\mathcal{A})  = D (\mathcal{A}) + P (\mathcal{A})  ,  
	\end{align}
	with 
	\begin{align}
		D (\mathcal{A})  = \bigg(\frac {2} {\pi \sqrt{\pi}} M T + O_{\vepsilon, A} \bigg(\frac {N^{3/2+\vepsilon}} {T^A} \bigg) \bigg) \sum_{ N < n \leqslant 2 N } |a_{n}|^2,  
	\end{align}
for any $A \geqslant 0$, and 
	\begin{align}\label{1eq: bound P(a)}
		P  (\mathcal{A}) \Lt  M T \sum_{q \Lt N/T}  \frac 1 {q} \int_{-M^{\vepsilon}/ M}^{M^{\vepsilon}/ M} \sum_{c \Lt N/ T q} \frac 1 {c}  \, \sumx_{   \valpha      (\mathrm{mod} \, c) } \bigg|   \sum_{ N < n \leqslant 2 N }  a_{n}   e \Big(   \frac { {\valpha}     n} {c} \Big)  e \bigg( \frac {n t} {c q} \bigg) \bigg|^2  \nd t .
	\end{align}  
\end{thm}

Our proof of Theorem \ref{thm: asymptotic large sieve} relies on careful analysis of the related Bessel integral from the Kuznetsov formula for $\SL_2 (\BZ)$, in particular the expression in Lemma \ref{lem: x>1}.  

The readers may compare Theorem \ref{thm: asymptotic large sieve} with Young's Theorem 7.1 in \cite{Young-GL(3)-Special-Points} and as well jump to \S \ref{sec: sketch} for a sketch of proof of Theorem \ref{thm G:(3)}.

\begin{remark}\label{rem: real assump}
Note that the square in {\rm\eqref{1eq: S, cusp}} is opened up into the double sum
$$  \mathop{\sum\sum}_{N < m, n \leqslant 2 N} a_m \overline{a}_{n} \lambda_{j} (m)\lambda_j (n) (m/n)^{it_j}  ,  $$ and,  in order to apply the Kuznetsov trace formula, this needs to be even in $t_j$, so the sequence $\mathcal{A} = \{a_n\}$ is forced to be real as assumed in Theorem {\rm\ref{thm: asymptotic large sieve}}.  However, in practice the sequence $\mathcal{A}$ is often {not} real-valued. Nevertheless, if one is concerned with bounds for $ S (\mathcal{A}) $ or $ S (\mathcal{A}) + T (\mathcal{A})  $, then one may always remove this assumption by  splitting $a_n$ into $\mathrm{Re}(a_n)$ and $\mathrm{Im}(a_n)$ at first, while splitting $\mathrm{Re}(a_n)$ or $\mathrm{Im}(a_n)$ into $a_n$ and $\overline{a}_n$ at the end. By the Cauchy inequality, one is reduced to estimating 
	\begin{align}
		\breve{D} (\mathcal{A})  =  M T \sum_{ N < n \leqslant 2 N } |a_{n}|^2, 
	\end{align}
\begin{align}\label{1eq: bound P(a), 2}
	\breve{P}  (\mathcal{A}) = M T \sum_{q \Lt N/T}  \frac 1 {q} \int_{-M^{\vepsilon}/ M}^{M^{\vepsilon}/ M} \sum_{c \Lt N/ T q} \frac 1 {c}  \, \sumx_{   \valpha      (\mathrm{mod} \, c) } \bigg|   \sum_{ N < n \leqslant 2 N }  a_{n}   e \Big(   \frac { {\valpha}     n} {c} \Big)  e \bigg( \frac {n t} {c q} \bigg) \bigg|^2  \nd t ,
\end{align} with the observation that both are invariant under $ a_n \ra \overline{a}_n$. 
\end{remark}




\delete{It follows that
\begin{align}\label{1eq: moments, Luo M=1}
	\sum_{T < t_j \leqslant T+1}  \left|L (s_j, \phi  \times u_j) \right|^2 \Lt  T^{1+\vepsilon},  
\end{align}  
and hence, in general, we have the short-interval extension of \eqref{1eq: mean Lindelof, T}: 
\begin{align}\label{1eq: moments, Luo M}
	\sum_{T < t_j \leqslant T+M}  \left|L (s_j, \phi  \times u_j) \right|^2 \Lt  M T^{1+\vepsilon},
\end{align} 
for any $1 \leqslant M \leqslant T$, if $\phi$ is a cusp form for $\GL_2$. }
	



\subsection{Comparison with Luo and Young's Approach} 

It is a remarkable observation of Luo \cite{Luo-Twisted-LS} that in the long-interval case $t_j \leqslant T$ there is an `Eisenstein--Kloosterman' cancellation that enabled him to improve Deshouillers and Iwaniec's \eqref{1eq: DI's bound}.  Presumably, one should expect such an effect to persist in the short-interval case of $T < t_j \leqslant T+M$. However, detecting the effect is a very subtle problem since in the shortest case $M = 1$   it would disappear as observed by  Luo  \cite{Luo-LS} (see \eqref{1eq: LS, M=1} and  \eqref{1eq: LS, M=1, twisted} below).  

As alluded to above, Luo \cite{Luo-Twisted-LS} detected the `Eisenstein--Kloosterman'  cancellation by the Euler--Maclaurin formula, while Young \cite{Young-GL(3)-Special-Points} did it by the Poisson summation formula, and then he continued with the $\GL_3$ Vorono\"i summation formula for further cancellation (actually, the dual sum is negligibly small). Young also noticed a curious similarity between this
problem and certain aspects of the large sieve inequality for $\Gamma_1 (q) \subset \mathrm{SL}_2 (\BZ)$ obtained in Iwaniec--Li \cite{Iwaniec-Li-Ortho}. 

Our approach differs in that we do not detect the  `Eisenstein--Kloosterman' cancellation between $ T (\mathcal{A})  $ and $P (\mathcal{A})$ as in \eqref{1eq: E, Eis} and \eqref{1eq: bound P(a)} (or \eqref{4eq: P(a), 2}) and apply directly the Vorono\"i summation to \eqref{1eq: bound P(a)}.  
A simple reason is that the Eisenstein contribution is already $O (  T^{3/2+\vepsilon})$: 
\begin{align}\label{1eq: mean Lindelof, GL(3), 2}
|L (1/2 , \phi) |^2 \cdot	\int_{\, T-M}^{T+M}    |L (1/2+2it , \phi)   |^2  \nd t \Lt_{\phi, \vepsilon}     T^{3/2 +\vepsilon},
\end{align} 
for all $1 \leqslant M \leqslant T$, 
as a result of the large sieve for Dirichlet polynomials \cite[Theorem 6.1]{Montgomery-Topics} (with $N $ up to $ T^{3/2+\vepsilon}$),  
\begin{align*}
	\int_{-T}^T \bigg|\sum_{ n \leqslant   N} a_n n^{it} \bigg|^2 \nd t \Lt (T+N) \sum_{ n \leqslant   N}  |a_n|^2, 
\end{align*} 
so its part canceled out should not exceed $ T^{3/2+\vepsilon} $ any way.



\delete{\footnote{ The bound {\rm\eqref{1eq: DI's bound}} was (7) in \cite{Luo-Twisted-LS} and cited as Theorem 6 of \cite{DI-Nonvanishing}. However, in the latter the sequence ${a}_{n}$ is {\it not} arbitrary but subject to certain conditions ($\mathrm{H}_1$),  ($\mathrm{H}_2$), and  ($\mathrm{H}_3$). Also note that {\rm\eqref{1eq: DI's bound}} has been used in Luo's proof of  {\rm\eqref{1eq: Luo's bound, 1}}  for the case $N \leqslant T$. 
		
		Nevertheless, if we choose $M = T$ in {\rm\eqref{1eq: large sieve}} and apply a dyadic summation, then  \eqref{1eq: DI's bound} may be improved slightly into
		\begin{align*}
			\sum_{t_j \leqslant T} \omega_j  \bigg| \sum_{ n \leqslant N}  a_{n} \lambda_{j} (n) n^{it_j} \bigg|^2 \Lt \big(T^2 +  T N \big) (TN)^{\vepsilon} \sum_{ n \leqslant N}  |a_{n}|^2. 
		\end{align*}
		At any rate, this put Luo's {\rm\eqref{1eq: Luo's bound, 1}} on solid ground in the case $N \leqslant T$. }}


\subsection{Remarks} 

The study of $L (s_j, \phi \times u_j)$   is of particular interest  if $\phi $ is a certain $\GL_2$ holomorphic cusp form, since the non-vanishing of these special $L$-values arises in the Phillips--Sarnak deformation theory of cusp forms \cite{Phillips-Sarnak}. See  \cite{Phillips-Sarnak,DI-Nonvanishing,DIPS-Maass,Luo-Non-Vanishing,Luo-Weyl,Luo-2nd-Moment}. 

This paper is an attempt to approach the subconvexity problem for  $L (s_j , \phi \times u_j)$ for $\phi$ a fixed Hecke--Maass form for $\GL_3$---note that its convexity bound   is attainable if we drop all but one term in \eqref{1eq: mean Lindelof, GL(3)}.  This problem is notoriously hard due to the `conductor drop': its $\gamma$-factor is of $\GL_3$  type of conductor $|s_j|^3$ but its Fourier coefficients are in the $ \GL_3 \times \GL_2$  Rankin--Selberg type. Nevertheless,  if $\phi$ were $\GL_2 $, the subconvexity for $L (s_j , \phi \times u_j)$ was achieved (as a special case) in the seminal work of Michel and Venkatesh \cite{Michel-Venkatesh-GL2}. 	

A subconvexity bound for $L (s_j, \phi  \times u_j)$ would be achieved once we could prove \eqref{1eq: mean Lindelof, GL(3)} for $M = T^{1/2-\delta}$ with some $\delta > 0$.  
However, it seems hard even to break the bound $ T^{3/2+\vepsilon} $ as in \eqref{1eq: mean Lindelof, GL(3), 2} for  $M = T^{1/2-\delta}$, although this was done for $ T^{1/2+\delta} \leqslant M \leqslant T^{57/70-\delta}$ in the recent work of Aggarwal,  Leung, and Munshi \cite{Munshi-A-L-GL(3)}: 
\begin{align*}
	\int_{\, T-M}^{T+M}    |L (1/2+it , \phi)   |^2  \nd t \Lt_{\phi, \vepsilon}     T^{ \vepsilon} \bigg(\frac {T^{9/4}} {M^{3/2}} + \frac {M^3} {T^{21/20}} + M^{7/4} T^{3/40} + M^{15/14} T^{15/28} \bigg). 
\end{align*}
Thus it is still important to analyze carefully the `Eisenstein--Kloosterman' cancellation in the short-interval case to see whether it is significant enough to break $T^{3/2+\vepsilon}$.

\subsection{Aside: Luo's Large Sieve on Short Intervals}

It was stated without proof by
Iwaniec \cite{Iwaniec-Spectral-Weyl} and proven independently by Luo \cite{Luo-LS} and Jutila \cite{Jutila-LS} that 
\begin{align}\label{1eq: LS, M=1}
	\sum_{T < t_j \leqslant T+1} \omega_j  \bigg| \sum_{ n \leqslant N}  a_{n} \lambda_{j} (n)  \bigg|^2 \Lt_{\vepsilon}  (T +  N  ) (TN)^{\vepsilon} \sum_{ n \leqslant N}  |a_{n}|^2 ,
\end{align}
while Luo observed that, by partial summation, \eqref{1eq: LS, M=1} is equivalent to its twisted variant:
\begin{align}\label{1eq: LS, M=1, twisted}
	\sum_{T < t_j \leqslant T+1} \omega_j  \bigg| \sum_{ n \leqslant N}  a_{n} \lambda_{j} (n) n^{it_j}  \bigg|^2 \Lt_{\vepsilon}  (T +  N  ) (TN)^{\vepsilon} \sum_{ n \leqslant N}  |a_{n}|^2 ;  
\end{align}
the twist  $n^{it_j}$ does not play a role because $t_j$ is restricted in a segment of unity length.

By a direct application of Young's large sieve in Lemma \ref{lem: Young's LS}  to the expression in \eqref{1eq: bound P(a)} in Theorem \ref{thm: asymptotic large sieve}, we recover   \eqref{1eq: LS, M=1, twisted} and hence provide the third   proof of  \eqref{1eq: LS, M=1}. 

\begin{cor}\label{cor: large sieve} 
	Let   $1 \leqslant M \leqslant T$. We have
	\begin{align}\label{1eq: large sieve}
		\sum_{ T < t_j \leqslant T+M }  \omega_j  \bigg| \sum_{     n \leqslant    { N} }  a_{n} \lambda_{j} (n) n^{it_j}  \bigg|^2  \Lt_{\vepsilon}  	M (  T+ N   ) (TN)^{\vepsilon}    \sum_{   n  \leqslant    {N} } |a_{n}|^2 , 
	\end{align}
	for any complex numbers $a_n$, where the implied constant depends on $\vepsilon$ only. 
\end{cor}

Actually, \eqref{1eq: large sieve} is equivalent to \eqref{1eq: LS, M=1, twisted} via dividing $(T, T+M]$ into $ M $ many intervals of unity length.

\subsection*{Notation}

By $X \Lt Y$ or $X = O (Y)$ we mean that $|X| \leqslant c Y$  for some constant $c  > 0$, and by $X \asymp Y$ we mean that $X \Lt Y$ and $Y \Lt X$. We write $X \Lt_{\valpha, \beta, ...} Y $ or $  X = O_{\valpha, \beta, ...} (Y) $ if the implied constant $c$ depends on $\valpha$, $\beta$, ....  

The notation $x \sim X$ stands for  $ X <  x \leqslant 2 X $ for $x$ integral or real according to the context.  

By `negligibly small' we mean $ O_A ( T^{-A} )$ for arbitrarily large but fixed $A > 0$. 

Throughout the paper,  $\vepsilon  $ is arbitrarily small and its value  may differ from one occurrence to another.

\delete{Moreover, Luo \cite{Luo-2nd-Moment}  proved the following asymptotic formula:
\begin{align}
	\sum_{j=1}^{\infty} \omega_j \left|L (s_j, \phi_2  \times u_j) \right|^2 \exp (-t_j / T) = c_0 T^2 \log T + c_1 T + O_{\phi_2, \vepsilon} (T^{11/6+\vepsilon}),  
\end{align}
where $c_0$ and $c_1$ are constants depending on $\phi_2$ only.}

\begin{acknowledgement}
	The author wishes to thank Wenzhi Luo, Matthew P. Young, and the referee for helpful comments and suggestions.  
\end{acknowledgement}

	\section{Preliminaries}
	
	\subsection{Exponential Sums}

Let $e (x) = \exp (2\pi i x)$. 	For integers $  m,  n , q$ and   $c  \geqslant 1$,      define 
	\begin{align}\label{1eq: defn Kloosterman}
		S   (m, n ; c ) = \sumx_{   \valpha      (\mathrm{mod} \, c) } e \bigg(   \frac {  \valpha     m +   \widebar{\valpha    } n} {c} \bigg) ,
	\end{align} 
	\begin{align} \label{2eq: defn V}
		V_{q} (m, n; c) =	\mathop{\sum_{\valpha     (\mathrm{mod}\, c)}}_{ (\valpha     (q-\valpha    ), c) = 1 }  e \bigg(   \frac {  \widebar{\valpha    } m +  \overline{q - \valpha    } n } {c} \bigg),
	\end{align}
	where the $\star$ indicates the condition $(\valpha    , c) = 1$ and $ \widebar{\valpha    }$ is given by $\valpha     \widebar{\valpha    }  \equiv 1 (\mathrm{mod} \, c)$. The definition of $V_{q} (m, n; c)$ is essentially from Iwaniec--Li  \cite[(2.17)]{Iwaniec-Li-Ortho}. Note that the Kloosterman sum $S (m,n;c)$ is real valued. Moreover, we have the Weil  bound:
	\begin{align}
			\label{2eq: Weil}
		S (m, n; c)   \Lt \tau (c) \sqrt{ (m, n, c) } \sqrt{c}, 
	\end{align} 
where as usual $\tau (c)$ is the number of divisors of $c$. 
It follows that 
\begin{align}
	\label{2eq: quad form, Kloosterman}  	   {\mathop{\sum \sum}_{  m, n  \leqslant    {N} } } \left| a_{m}  \overline{a }_{n}   S (m, n; c) \right|   \Lt \tau^2 (c) \sqrt{c} N  \sum_{   n \leqslant    {N}  } |a_{n}|^2, 
\end{align} 
	for any complex $a_n$. 
	
	\begin{lem}\label{lem: S = V}
		 We have 
		 \begin{align}\label{2eq: S = V}
		 	 S (m, n; c)  e\Big(\frac {m+n} {c} \Big) = \sum_{qr = c } V_q (m, n; r). 
		 \end{align}
	\end{lem}
	
\begin{proof}
For the reader's convenience, we record here Luo's proof  as in \cite[\S 3]{Luo-LS}\footnote{Note that there is a typo in  \cite{Luo-LS}:  the summation  is missed.}.  By \eqref{1eq: defn Kloosterman} we write 
	\begin{align*}
		S (m, n; c)  e\Big( \frac {m+n} {c} \Big) = \sumx_{   \valpha      (\mathrm{mod} \, c) } e \bigg(   \frac {  (1-\valpha    ) m +  (1- \widebar{\valpha    }) n } {c} \bigg),
	\end{align*}
	and split the sum according to  $(1-\valpha    , c) = q$. Thus $c = q r$ and $\valpha     = 1 - \widebar{\beta} q$, where $\beta$ ranges over residue classes modulo $r$  such that $  (\beta (q-\beta), r) = 1$. We obtain
	\begin{align*}
		S (m, n; c)  e\Big(\frac {m+n} {c} \Big) =   \sum_{qr=c} \mathop{\sum_{\beta (\mathrm{mod}\, r)}}_{ (\beta (q-\beta), r) = 1 }  e \bigg(   \frac {  \widebar{\beta} m +  \overline{q - \beta} n } {r} \bigg) , 
	\end{align*}
as desired. 
\end{proof}
	
	\subsection{Kuznetsov Trace Formula for {\protect\scalebox{1.06}{$\SL_2 (\BZ)$}}} 
	
	Let $ \{u_j (z)\}_{j=1}^{\infty}$ be an orthonormal basis of Hecke--Maass forms for $\mathrm{SL}_2 ( \BZ)$. For each $u_j (z)  $ with Laplacian eigenvalue $\lambda_j = 1 / 4 + t_j^2$ ($t_j > 0$), it has Fourier expansion of the form
	\begin{align*}
		u_j(z)=   \sqrt{ y}  \sum_{n\neq 0}\rho_j(n) K_{it_j}(2\pi|n|y)e(nx) .
	\end{align*}
As usual, write $s_j = 1/2+ i t_j$ so that $\lambda_j = s_j (1-s_j)$. 	Let $\lambda_j (n)$ ($n \geqslant 1$) be its Hecke eigenvalues. It is well known that $ \lambda_j (n)$ are all real.  We may assume  $ u_j (z)$ is even or odd in the sense that $ u_j (- \widebar{z}) = \epsilon_j u_j  (z)$ for $\epsilon_j = 1 $ or $-1$.    
Then   $\rho_j (\pm n) =   \allowbreak \rho_j (\pm 1)  \lambda_j (n) $, while $\rho_j (-1) = \epsilon_j \rho_j (1)$. 

Now we state the Kuznetsov trace formula as in \cite[Theorem 1]{Kuznetsov}. 

	\begin{lem}\label{lem: Kuznetsov}
	Let  $h (t)$ be an even function satisfying the conditions{\hspace{0.5pt}\rm:}
	\begin{enumerate} 
		\item[{\rm (i)\,}] $h (t)$ is holomorphic in  $|\operatorname{Im}(t)|\leqslant {1}/{2}+\vepsilon$,
		\item[{\rm (ii)}] $h(t)\Lt (|t|+1)^{-2-\vepsilon}$ in the above strip. 
	\end{enumerate}
	Then for $m, n \geqslant  1$ 	we have the following  identity{\hspace{0.5pt}\rm:} 
	\begin{equation}\label{2eq: Kuznetsov}
		\begin{split}
			\sum_{j = 1}^{\infty}  \omega_j h(t_j)   \lambda_j(m)   \lambda_j(n)  + \frac{1}{\pi} & \int_{-\infty}^{\infty} \omega(t) h(t) (n/m)^{i t} \sigma_{2it}(m)\sigma_{-2it}(n) \nd t\\
			&= \delta_{m, n} \cdot H +   \sum_{c= 1}^{\infty} \frac{S(m, n;c)}{c} H\bigg(\frac{4\pi\sqrt{m n}}{c}\bigg),
		\end{split}
	\end{equation}
	where  $\delta_{m, n}$ is the Kronecker $\delta$-symbol,  $S (m, n; c)$ is the Kloosterman sum, and 
	\begin{align}
		 \sigma_{\vnu} (n) = \sum_{d | n} d^{\hspace{0.5pt} \vnu},  
	\end{align}
	\begin{equation}\label{2eq: omega}
		\omega_j=\frac{ |\rho_j(1)|^2}{\cosh(\pi t_j)}, \qquad \omega (t) = \frac {1} {|\zeta(1+2it)|^2},
	\end{equation}
	\begin{align}\label{2eq: integrals} 
			  H =\frac{1}{\pi^2}\displaystyle\int_{-\infty}^{\infty}h(t)\tanh(\pi t)t \nd t , \qquad  
			  H (x)= \frac {2i} {\pi}   \int_{-\infty}^{\infty} J_{2it} (x) h (t) \frac {t \nd t} {\cosh (\pi t )}. 
		\end{align}  
\end{lem}
	
		The harmonic weights $\omega_j$ and $\omega (t)$ play a very minor role in our problem as
	\begin{align}\label{2eq: lower bound of omega}
		\omega_j\Gt t_j^{-\vepsilon}, \qquad \omega(t) \Gt t^{-\vepsilon} ; 
	\end{align}
	see \cite[Theorem 2]{Iwaniec-L(1)} and \cite[Theorem 5.16]{Titchmarsh-Riemann}, respectively.

	\subsection{Stationary Phase}

	We record here  \cite[Lemma A.1]{AHLQ-Bessel}, a slightly improved version of  \cite[Lemma {\rm 8.1}]{BKY-Mass}.

	\begin{lem}\label{lem: staionary phase}
		Let $\varww   \in C_c^{\infty} (a, b)$. Let  $f  \in C^{\infty} [a, b]$ be real-valued.  Suppose that there
		are   parameters $P, Q, R, S, Z  > 0$ such that
		\begin{align*}
			f^{(i)} (x) \Lt_{ \, i } Z / Q^{i}, \qquad \varww^{(j)} (x) \Lt_{ \, j } S / P^{j},
		\end{align*}
		for  $i \geqslant 2$ and $j \geqslant 0$, and
		\begin{align*}
			| f' (x) | \Gt R. 
		\end{align*}
		Then 
		\begin{align*}
			\int_a^b  e (f(x)) \varww (x)  \nd x \Lt_{ A} (b - a) S \bigg( \frac {Z} {R^2Q^2} + \frac 1 {R Q} + \frac 1 {R P} \bigg)^A  
		\end{align*} 
		for any  $A \geqslant 0$.
	\end{lem}

According to \cite{KPY-Stationary-Phase}, let us introduce the notion of inert functions in a simplified setting.  

\begin{defn}\label{defn: inert}
Let $\boldsymbol{I}  \subset \BR_+^{d}$ be a product of intervals {\rm(}not necessarily finite{\rm)}.  For $X \geqslant 1$, we say a smooth function $\varww \in C^{\infty} (\boldsymbol{I})$ is $X$-inert if 
\begin{align*}
\boldsymbol{x}^{\boldsymbol{i}} 	\varww^{(\boldsymbol{i})} (\boldsymbol{x})  \Lt_{\boldsymbol{i}} X^{|\boldsymbol{i}|}  , \qquad \text{($\boldsymbol{x} \in \boldsymbol{I}$),} 
\end{align*} for every $\boldsymbol{i} \in \mathbb{N}_0^{d}$, where in the multi-variable notation $\boldsymbol{x}^{\boldsymbol{i}} = x_1^{i_1} \cdots x_d^{i_d}$, $ \varww^{(\boldsymbol{i})} (\boldsymbol{x}) = \varww^{(i_1, \cdots, i_d)} (x_1, \cdots, x_d) $, and $|\boldsymbol{i}| = i_1+ \cdots + i_{d}$. 
\end{defn}

Next, we record here  a generalization of 
the stationary phase estimate in \cite[Theorem 1.1.1]{Sogge}.

\begin{lem}\label{lem: stationary phase estimates}
	Let  $ \sqrt {\lambda} \geqslant X \geqslant  1$.  	Let $\varww (x, \lambda, \boldsymbol{x})  \in C^{\infty} ((a, b) \times [X^2, \infty) \times \boldsymbol{I})$ be $ X$-inert, with compact support in the first variable $x$. Let $f (x) \in C^{\infty} [a, b]$ be   real-valued.  Suppose    $f(x_0) = f'(x_0) = 0$ at a point  $  x_0 \in (a, b)$, with $ f'' (x_0) \neq 0$ and $f' (x) \neq 0$ for all $x \in [a, b] \smallsetminus \{x_0\} $. Define
	\begin{align*}
		I (\lambda, \boldsymbol{x}) = \int_a^b e (\lambda f(x)) \varww (x, \lambda, \boldsymbol{x})  \nd x,
	\end{align*}  
then $\sqrt{\lambda} \cdot I (\lambda, \boldsymbol{x}) $ is an $X$-inert function.  
\end{lem}

\begin{proof}
Note that if there were no variable $\boldsymbol{x}$, then this lemma is \cite[Theorem 1.1.1]{Sogge} in the case $X = 1$ and  \cite[Lemma 7.3]{Qi-GL(3)} in general.   However, since $ e ( \lambda f(x))$ does not involve $\boldsymbol{x}$, the derivatives for the added variable  $\boldsymbol{x}$ may be treated easily. 
\end{proof}

 Of course, the main theorem in \cite{KPY-Stationary-Phase} is much more general than Sogge's \cite[Theorem 1.1.1]{Sogge}, but the latter has a simpler proof and no error term.  

The next lemma is a simple application of Lemmas \ref{lem: staionary phase} and \ref{lem: stationary phase estimates}. 

\begin{lem} \label{lem: analysis of integral}
	Let  $ \gamma > 1$. 
	For $ \sqrt {\lambda} \geqslant X \geqslant  1$ and $\rho > 0$, define 
	\begin{align*}
		I_{\gamma}^{\pm} (\lambda, \boldsymbol{x}) =   \int_{\rho}^{2\rho}  e \big(\lambda \big(x \pm \gamma   x^{1/\gamma} \big) \big) \varww (x, \lambda, \boldsymbol{x}) \nd x,  
	\end{align*} 
	for an $X$-inert function  $\varww (x, \lambda, \boldsymbol{x}) \in C^{\infty} ([\rho ,  2 \rho] \times [X^2, \infty) \times \boldsymbol{I} )$, with compact support in the first variable $x$. 
	
	{\rm\,(i)} We have
	$$ I_{\gamma}^{\pm} (\lambda, \boldsymbol{x}) \Lt_A  \rho \cdot \bigg( \frac {  X  } {\lambda (\rho +\rho^{1/\gamma})}\bigg)^A  $$ 
	for any value of $\rho$ in the $+$ case, or for $ \min \big\{ \rho/\sqrt{2}, \sqrt{2}/\rho \big\}   < 1 / 2 $ in the $-$ case.  
	
	{\rm(ii)} Define  
	\begin{align*}
		\varvv_{\gamma} (\lambda, \boldsymbol{x} ) =  e (  \lambda (\gamma -1) ) \cdot \sqrt{\lambda}   I_{\gamma}^{-} (\lambda, \boldsymbol{x}  ), 
	\end{align*}  
	then $\varvv_{\gamma} (\lambda, \boldsymbol{x} )$ is an $X$-inert function for any $1/2 \leqslant \rho/\sqrt{2} \leqslant 2 $. 
\end{lem}

\begin{remark}
	It will be used implicitly the fact that the implied constants in the bounds for the derivatives of $ \varvv_{\gamma} (\lambda, \boldsymbol{x} ) $ depend uniformly on those for $ \varww (x, \lambda, \boldsymbol{x}) $. 
\end{remark}

\subsection{Hybrid Large Sieve of Young}  
Let $\gamma \neq 0$,  $\tau, v   > 0$,  and $C, N \Gt 1$.  The following hybrid large sieve inequality is a special case of Young's Lemma 6.1 in \cite{Young-GL(3)-Special-Points},  
	\begin{align}\label{2eq: hybrid ls, Young}
		\int_{-\tau}^{\tau}   \sum_{c \shskip \leqslant C}  \, \sumx_{   \valpha      (\mathrm{mod} \, c) }   \bigg|  \sum_{ n  \sim N }    a_{n}   e \Big(   \frac {\valpha     n} {c} \Big)  e  \bigg(   \frac { n^{\gamma} t} { v }  \bigg)   \bigg|^2    \nd t \Lt_{\gamma}   \big(\tau C^2  + v N^{1-\gamma} \big)   \sum_{ n  \sim N }   |a_n|^2. 
	\end{align}
The variant as follows will be more convenient for our applications. 

\begin{lem}\label{lem: Young's LS}
We have 
	\begin{align}\label{2eq: hybrid ls, Young, 2}
		\begin{aligned}
			\int_{-\tau}^{\tau} \hspace{-1pt} \sum_{c \shskip \leqslant C}  \frac 1 {c} \, \sumx_{   \valpha      (\mathrm{mod} \, c) } \hspace{-1pt} \bigg| \hspace{-1pt}    \sum_{ n  \sim N}    a_{n}   e \Big(   \frac {\valpha     n} {c} \Big)  e  \bigg(   \frac { n^{\gamma} t} {c v} \bigg) \hspace{-1pt} \bigg|^2 \hspace{-1pt}  \nd t \Lt_{\gamma} \hspace{-1pt} \big(\tau C + v N^{1-\gamma} \log C \big) \hspace{-2pt}   \sum_{ n \sim N }    |a_n|^2,
		\end{aligned}
	\end{align}
for any complex $a_n$.  
\end{lem}
\begin{proof}
	By the change $t \ra c t$ we rewrite the expression on the left  as
	\begin{align*}
		\sum_{   c \shskip \leqslant   C }  \,	 \int_{- \tau/ c}^{\tau/ c}   \ 	\sumx_{\valpha     (\mod c)}  \bigg| \sum_{   n \sim N }   a_{n}   e \Big(   \frac {\valpha     n} {c} \Big) e     \bigg(   \frac { n^{\gamma} t} { v } \bigg) \bigg|^2  \nd t , 
	\end{align*}  then via a dyadic partition \eqref{2eq: hybrid ls, Young, 2} follows easily from  \eqref{2eq: hybrid ls, Young}. 
\end{proof}

\subsection{Maass Forms for {\protect\scalebox{1.06}{$\SL_3 (\BZ)$}}}  We refer the reader to Goldfeld's book \cite{Goldfeld} for the theory of Maass forms for $\SL_3 (\BZ)$.  

Let $\phi$ be a Hecke--Maass   form for $\SL_3 (\BZ)$ of Fourier coefficients $ A (m    , n    )  $, normalized so that $ A(1,1) = 1 $, and Langlands parameters $ \{ \lambda    _1, \lambda    _2, \lambda    _3 \} $, with $\lambda    _1+\lambda    _2+\lambda    _3 = 0$. The dual Maass form $\widetilde{\phi}$ has Fourier coefficients $ A(n    , m    ) = \overline{A(m    , n    )}$ and Langlands parameters $ \left\{ -\lambda    _1, -\lambda    _2, -\lambda    _3 \right\} =  \left\{ \overline{\lambda}_1, \overline{\lambda}_2, \overline{\lambda}_3 \right\}$. For later use, we record here the  Rankin--Selberg estimate:
\begin{equation}\label{2eq: RS}
	\mathop{\sum\sum}_{m^2 n \shskip \leqslant X}|A(m, n)|^2 \Lt  X ,
\end{equation}  
together with the Hecke relation
\begin{align*}
	A (m, n) = \sum_{d|(m,n)} \mu (d ) A(m/d, 1) A(1, n/d),  
\end{align*}
we deduce by  the Cauchy inequality that 
\begin{equation}\label{2eq: RS, 2}
	\sum_{m \leqslant X} \sum_{n \leqslant Y}    {|A(m, n)|^2}   \Lt  (XY)^{1+\vepsilon}. 
\end{equation} 
\delete{Moreover, it is known by   Kim--Sarnak \cite{Kim-Sarnak} that 
\begin{align}\label{2eq: Kim-Sarnak}
	|\mathrm{Re} (\lambda_1)| , \,	|\mathrm{Re} (\lambda_2)| , \,	|\mathrm{Re} (\lambda_3)| \leqslant \frac {5} {14}.  
\end{align}}

As in \cite[\S 6.5]{Goldfeld} or \cite[\S 6]{Miller-Schmid-2006}\footnote{It is slightly inconsistent that  $ L_{{\phi}} (s) $ in \cite[\S 6.5]{Goldfeld} or  $L (s, {\phi})$ in \cite[\S 6]{Miller-Schmid-2006} is  $ L_{\widetilde{\phi}} (s) $ in \cite[\S 9.4]{Goldfeld}. },  define the $L$-function attached to $\phi$ by
\begin{align}
	L(s,\phi)=\sum_{n=1}^{\infty}\frac{A(1,n)}{n^s},
\end{align} 
for $\mathrm{Re} (s) > 1$, and by analytic continuation for all $s$ in the complex plane. The $\gamma$-factor of $\phi$ is equal to
\begin{align}\label{2eq: gamma (s, phi)}
	\gamma (s, \phi) = \pi^{-3s/2} \Gamma \bigg(\frac {s+\lambda    _1}2 \bigg) \Gamma \bigg(\frac {s+\lambda    _2}2 \bigg) \Gamma \bigg(\frac {s+\lambda    _3}2 \bigg). 
\end{align}
The functional equation for $ L(s,\phi) $ reads
\begin{align}\label{2eq: FE, GL(3)}
	\gamma (s, \phi) L(s, \phi) = \gamma (1-s, \widetilde{\phi})  L (1-s, \widetilde{\phi}) . 
\end{align}

{\begin{remark}
For $\phi$ of type $ (\vnu_1, \vnu_2)$, its Langlands parameters are given by 
\begin{align*}
	\lambda    _1 = 1 - 2 \vnu_1 - \vnu_2   , \qquad \lambda    _2 =   \vnu_1 - \vnu_2, \qquad \lambda    _3 =  - 1 + \vnu_1 +2 \vnu_2 . 
\end{align*}
\end{remark}
}

\subsection{Rankin--Selberg $L$-function {\protect\scalebox{1.06}{$L (s, \phi \times u_j)$}}} Define
\begin{align}\label{2eq: defn L(s)}
	L (s, \phi \times u_j) = \sum_{m    =1}^{\infty} \sum_{n    =1}^{\infty}\frac{A(m    ,n    )\lambda_j (n    )}{(m    ^2 n    )^s},  
\end{align}
for $\mathrm{Re} (s) > 1$, and it admits analytic continuation to the whole complex plane.  
Let us introduce $\delta_j = 0$ or $1$ according as $u_j$ is even or odd, and define the $\gamma$-factor
\begin{align}\label{2eq: gamma (s, phi uj)}
	\gamma (s, \phi \times u_j) =   \gamma (s + \delta_j - i t_j, \phi  ) \gamma (s + \delta_j + i t_j, \phi  ). 
\end{align}
\delete{so that $\gamma (s, \phi \times u_j) = \gamma_{\delta_j} (s, t_j, \phi)$ if we define
\begin{align}
	\gamma_{\delta} (s, t, \phi) = \gamma (s + \delta - t, \phi  ) \gamma (s + \delta + t, \phi  ). 
\end{align}}
Then the functional equation for $ L(s,\phi \times u_j) $ reads
\begin{align}\label{2eq: FE, GL(3) x GL(2)}
	\gamma (s, \phi \times u_j) L(s, \phi \times u_j) = \epsilon_j \gamma (1-s, \widetilde{\phi} \times u_j)  L (1-s, \widetilde{\phi} \times u_j) . 
\end{align}

\subsection{Vorono\"i Summation Formula for {\protect\scalebox{1.06}{$\SL_3 (\BZ)$}}} 
 
The Vorono\"i summation formula for $\SL_3 (\BZ)$ was established by Miller and Schmid \cite{Miller-Schmid-2006}. However,  we propose here to use the version of Miller and Zhou \cite{MZ-Voronoi}\footnote{The Vorono\"i summation formula for $\GL_N $ in \cite{MZ-Voronoi} is normalized (with only an extra factor $1/|y|$ in the Hankel transform) so that it coincides with the classical Vorono\"i summation formula  for $\GL_2$ and   Poisson summation formula for $\GL_1$.}. 

\begin{lem}\label{lem: Voronoi}
	
	For $\omega  \in C_c^{\infty} (0, \infty)$ define its Hankel transform $\Omega $ by
	\begin{align*}
		\Omega (\pm y) =  \frac{1}{4 \pi i}\int_{(-1)}G^{\pm}(s)\widetilde{\omega}(s) y^{s-1}\nd s  ,
	\end{align*}
	\footnote{Note that the $1/2\pi i$ in  (7) of \cite{MZ-Voronoi} should be $1/4\pi i$ as in  (1.7) of \cite{Miller-Schmid-2009}.}where  $\widetilde{\omega} (s)$ is the Mellin transform of $\omega (x) $, and
	\begin{equation*}
		G^{\pm}(s)=\frac{\gamma(1-s, \widetilde{\phi})}{\gamma(s, {\phi})}\pm \frac 1 {i^3} \frac{\gamma(2-s, \widetilde{\phi})}{\gamma(1+s, {\phi})} . 
	\end{equation*}
	Let $\valpha, \widebar{\valpha}, c, m   $ be integers with   $\valpha \widebar{\valpha} \equiv 1 (\mathrm{mod}\, c)$ and $c, m   > 0$.	Then    we have 
	\begin{equation}
		\begin{split}
		   \sum_{n=1}^{\infty}    {A(m   ,n)}  e\bigg(  \frac{\widebar{\valpha}n}{c}\bigg) \omega (n)  =     \sum_{\pm}   \sum_{d\mid c m   }  d       \sum_{n =1}^{\infty}   A(n , d)  \frac  {S\left(\pm n,  {\valpha}m    ;c m   /d\right)} {c^2 m }  \Omega \bigg( \hspace{-2pt} \mp \frac{d^2 n} {c^3 m   } \bigg) .  
		\end{split}
	\end{equation}  
\end{lem}


	
	According to \cite[\S \S 3.3, 14]{Qi-Bessel}, there is a Bessel kernel $J_{\phi} (x)$  attached to $\phi$  so that the Hankel transform may indeed be realized as an integral transform
\begin{align}\label{2eq: Hankel}
	\Omega (y) = \int_{0}^{\infty} \omega (x) J_{\phi} ( - x y) \nd x,
\end{align}
and  the following asymptotic expansion holds:
\begin{align}\label{4eq: asymptotic, Bessel, R} 
	J_{\phi}  (\pm x ) & =     \frac { { e  \left(   \pm  3 x^{1/3}  \right)   }}  {x^{1/3} }   \sum_{k= 0}^{K-1} \frac {B^{\pm}_{k } }  {x^{  k/3  }}   +  O  \bigg( \frac 1 {x^{ (K + 1) / 3}} \bigg), 
\end{align}
for $x \Gt  1$, where $B^{\pm}_{k }$ are some constants depending on the Langlands parameters of $\phi$.

	\section{Analysis for the Bessel Integral}

Subsequently, we shall always assume $ T^{\vepsilon} \leqslant M \leqslant T^{1-\vepsilon} $. 
Let us write  $h (t)$ defined by \eqref{1eq: weight k} as follows: 
\begin{align}\label{3eq: defn h(t)}
	h (t) = \upbeta \bigg(     \frac {t - T } {M }    \bigg) \hspace{-1pt} + \upbeta \bigg(     \frac {t + T } {M }    \bigg), \qquad  \upbeta (r) = \exp \big( \hspace{-1pt}    - r^2\big),
\end{align}
and   define
\begin{equation}\label{3eq: defn h(t; y)}
	h( t; y ) =  h (t)  	\cos  ( 2  t \log y  ).  
\end{equation}
Note that  $h( t; y )$ is even in $t$ as required by the Kuznetsov trace formula.   
The purpose of this section is to study its Bessel integral
	\begin{align}
		H (x, y) = \frac{2 i } {\pi} \int_{-\infty}^{\infty} J_{2it} (x) h (t; y) \frac {t \nd t} {\cosh (\pi t) }. 
	\end{align}
Some preliminary analysis will be similar to that in Xiaoqing Li's work  \cite{XLi2011}. 

\subsection{} 

For any non-negative integer $A$,  let $2A + 1 < 2\delta < 2A +3 $.  By contour shift,  
\begin{align*}
	H    (x, y)   =    	\frac 2 {\pi i}    \sum_{k = 0}^{A}   (-1)^k (2k+1) \cdot    J_{2k+1} (  x) h    (- (k+1/2) i; y)  & \\
	  - \frac 2 {\pi i}  \int_{-\infty}^{\infty}   {J_{  2it + 2 \delta } (x)  }  h    (  t - \delta i   ; y  )    \frac {   t - \delta i    } {\cos  (\pi  ( i t +  \delta   )  )} \nd t &.
\end{align*}
By  the Poisson  integral representation
(see  \cite[3.3 (5)]{Watson}): 
\begin{align*}
	J_{\vnu} (x) = \frac {  ( x/2  )^{\vnu}} {\sqrt{\pi} \Gamma   (\vnu +     1 / 2  )  } \int_0^{\frac 1 2 \pi}   \cos (x \cos \theta ) \sin^{2 \vnu} \theta \,  {\nd}   \theta, \qquad  \mathrm{Re}  (\vnu)  > - \frac 1 2, 
\end{align*} 
along with the Stirling formula,
we infer that  
\begin{align*}
	J_{2k+1} ( x)  \Lt  x^{2k+1}, \qquad 
	\frac {J_{  2it + 2 \delta } (x)  } {\cos  (\pi  ( i t +  \delta   )  ) }   \Lt_{\delta}    \lp \frac { x  } { |t| + 1   } \rp^{2 \delta}  . 
\end{align*}  
Consequently, if we write
\begin{align*}
	u =   {xy} + x/y, 
\end{align*} then it follows from  \eqref{3eq: defn h(t)} and \eqref{3eq: defn h(t; y)} that 
\begin{align*}
	H (x, y) \Lt \upbeta (T/M)  \sum_{k=0}^{A} u^{2k+1}  +   \frac { M   u^{2 \delta} } {T^{2 \delta -1}} \Lt  \frac { M   u } {T^{2 A }}  ,
\end{align*} 
provided $u \Lt 1$.

\begin{lem}\label{lem: x<1}
Let $u =   {xy} + x/y$. Then for $u \Lt 1$, we have $H (x, y) = O_{A}  ( M  u / T^{2A}  )$ for any integer $A \geqslant 0$. 
	
\end{lem}

\subsection{} 

We start with the Mehler--Sonine integral as in [Wat, 6.21 (12)]: 
\begin{align*}
	J_{  \vnu} (x) = \frac 2 {\pi} \int_0^{\infty} \sin (x \cosh r - \vnu \pi/2) \cosh ( \vnu r) \nd r, \qquad \text{($|\mathrm{Re} (\nu)| < 1$)}, 
\end{align*}
so that
\begin{align*}
	\frac{J_{2it} (x) - J_{-2it} (x)} {\cosh (\pi t)}
	=   \frac {2} {\pi i} \tanh (  \pi t) \int_{-\infty}^{\infty} \cos (x \cosh r) \cos (2  t r) \nd r. 
\end{align*}
For $x \Gt 1$, it follows by partial integration that only a negligibly small error will be lost if the integral
above is truncated at $|r| = T^{\vepsilon}$.  Therefore, up to a  negligibly small error $ 	H (x, y) $ is equal to
\begin{align*}
  \frac{4 } {\pi^2}  \int_{0}^{\infty} t \upbeta ((t-T)/M) \cos (2 t \log y)  \tanh (\pi t) \int_{-T^{\vepsilon} }^{T^{\vepsilon}}   \cos (x \cosh r) \cos (2  t r) \nd r   {  \nd t} . 
\end{align*}
Next we change the order of integration,  remove the factor $\tanh (\pi t)$ as $   \tanh (\pi t) = 1 + O (  \exp (-2\pi t))$ ($t > 0$), and extend the  $t$-integrals onto $(-\infty, \infty)$, then, again up to a  negligible error, this is simplified into
\begin{align*}
	\frac{4 } {\pi^2}  \int_{-T^{\vepsilon} }^{T^{\vepsilon}}  \cos (x \cosh r) \int_{-\infty}^{\infty} t \upbeta ((t-T)/M) \cos (2t (r-\log y))  {  \nd t} \nd r  . 
\end{align*} 
On the change of variables  $r \ra r + \log y$ and $t \ra T + M t$, this integral   turns into the sum of 
\begin{align*}
	\frac{4 M T} {\pi^2} \int_{-T^{\vepsilon} }^{T^{\vepsilon}}  \cos (x \cosh (r+\log y)) \int_{-\infty}^{\infty} \upbeta (t)   \cos (2 T r + 2  M t r )  {  \nd t} \nd r  , 
\end{align*}
and
\begin{align*}
	\frac{4 M^2} {\pi^2}  \int_{-T^{\vepsilon} }^{T^{\vepsilon}}  \cos (x \cosh (r+\log y)) \int_{-\infty}^{\infty} t \upbeta (t)     \cos (2 T r + 2  M t r )  {  \nd t} \nd r .
\end{align*}
It is easy to show that the phase 
\begin{align*}
	x \cosh (r+\log y) = \frac 1 2 \cosh r \cdot \bigg(xy + \frac x y \bigg) + \frac 1 2 \sinh r \cdot \bigg(xy - \frac x y \bigg) , 
\end{align*}
and, as $ \upbeta (t)  $ is even,  the inner integrals 
\begin{align*}
	\int_{-\infty}^{\infty} \upbeta (t)   \cos (2 T r + 2  M t r )  {  \nd t}  
	 =   {\sqrt{\pi}}   \upbeta (Mr)   \cos ( 2 T r) ,
\end{align*}
\begin{align*}
	\int_{-\infty}^{\infty} t \upbeta (t)   \cos (2 T r + 2 M t r )  {  \nd t} 
	  =  - {\sqrt{\pi}}  M r   \upbeta (Mr)   \sin ( 2 T r) , 
\end{align*}
by simple trigonometric calculations and applications of \cite[3.896 4, 3.952 1]{G-R}. Note that $\upbeta' (r) = -2r \upbeta (r)$, so the integrals above become 
\begin{align*} 
		\frac{4 M T} {\pi \sqrt{\pi}} \int_{-T^{\vepsilon} }^{T^{\vepsilon}} \upbeta (Mr) \cos (2 Tr)  \cos (2 ((v+w) \cosh r + (v-w) \sinh r))   \nd r  ,  
\end{align*}
and
\begin{align*} 
 	\frac{2 M^2} {\pi \sqrt{\pi}}  \int_{-T^{\vepsilon} }^{T^{\vepsilon}}  \upbeta' (Mr)  \sin (2 Tr)  \cos (2((v+w) \cosh r + (v-w) \sinh r))  \nd r  ,  
\end{align*}
if we set
\begin{align*}
	v = \frac {xy} {4}, \qquad w = \frac {x/y} {4}. 
\end{align*}
Since $ \upbeta (Mr) $ and $ \upbeta' (Mr)$ are of exponential decay, the integrals may be effectively truncated at $ |r| = M^{\vepsilon} / M $. 
So far, we have established the following integral formula for $H ( x, y) $.  

\begin{lem}\label{lem: x>1}
For $x \Gt 1$, we have the expression
	 \begin{align}\label{4eq: H = I}
	 \begin{aligned}
	 		H ( x, y) = 	 \mathrm{Re} \big\{  \exp (2i (v+w)) I (v, w)  \big\}  
	 	+ O_{A}   (T^{-A}   )  ,  \qquad v = \frac {xy} {4}, \quad w = \frac {x/y} {4}, 
	 \end{aligned}
	 \end{align}
	 for any $A \geqslant 0$, with
	 \begin{align}\label{4eq: I}
	 I(v, w) =	M T \int_{-M^{\vepsilon}/ M}^{M^{\vepsilon}/M} g (r )  \exp (2 i   (v \uprho_{+}  (r) - w \uprho_{-}  (r)) )  \nd r,
	 \end{align}
 in which 
	 \begin{align}\label{3eq: defn of g(r)}
	 g (r ) =	\frac {2 } {{\pi \sqrt{\pi}}} \big(2 \upbeta (Mr) \cos (2Tr) + M/T \cdot \upbeta' (Mr) \sin (2Tr) \big), 
	 \end{align}
 \begin{align}\label{4eq: rho (r)}
 	\uprho_{+} (r) = \sinh r + \cosh r - 1, \qquad \uprho_{-} (r) = \sinh r - \cosh r + 1. 
 \end{align}
\end{lem}

\delete{In practice, 
\begin{align}
	x =  \frac {4\pi\sqrt{mn}} {c}, \qquad y = \sqrt{\frac m n}, 
\end{align}
and hence
\begin{align}\label{3eq: v & w}
	v =  \pi \frac {m } {c} , \qquad w =  \pi \frac {n} {c} .  
\end{align}
It is a pleasant  feature that $v$ and $w$ become separate inside the integral in \eqref{4eq: I}.  }

 
 \subsection{Stationary Phase Analysis for  {\protect\scalebox{1.06}{$ I (v, w) $}}} 
 
Finally,  we analyze the integral $I       (v, w)$ in \eqref{4eq: I}  by Lemma   \ref{lem: staionary phase}.

\begin{lem}\label{lem: analysis of I}
	Let   $T^{\vepsilon} \leqslant M \leqslant T^{1-\vepsilon}$.  
Then $I       (v, w) = O_A (T^{-A})$ if $v, w \Lt T$. 
\end{lem}


\begin{proof}
		Define $$f_{\pm} (r) =  \pm T r + v \uprho_{+}  (r) - w \uprho_{-}  (r)  ,  $$ so that $ 2 f_{\pm} (r)$ are the phase functions of the integral $I(v, w) $ defined by \eqref{4eq: I}, \eqref{3eq: defn of g(r)}, and \eqref{4eq: rho (r)}.  By \eqref{4eq: rho (r)}, 
	\begin{align*}
		f_{\pm}' (r) =     \pm   T +     (v+w) \sinh r + (v-w) \cosh r   .
	\end{align*}
On the range $|r| \leqslant M^{\vepsilon}/ M$, we  have $ |f'_{\pm} (r)| \Gt T$ and $ f^{(i)}_{\pm} (r) \Lt v + w  $ for any $i \geqslant 2$. By applying Lemma \ref{lem: staionary phase}, with $P = 1/M$, $Q = 1$,  $Z = v + w$, and $ R = T$,  we infer that the integral  $I       (v, w) $ is negligibly small.   
	\end{proof}
  
		\delete{\begin{remark}\label{rem: M, T}
			Note that  $ T^{\vepsilon} \leqslant M \leqslant T^{1-\vepsilon} $ is crucial in the proof of Lemma \ref{lem: analysis of I}.   However, in Corollary \ref{cor: large sieve},  
				this assumption may be easily relaxed into $ 1 \leqslant M \leqslant T$ at the cost of only $T^{\vepsilon}$.  
		\end{remark}}

\section{The Special Twisted Large Sieve} \label{sec: large sieve}


Let $\mathcal{A} = \{a_{n} \}$ be a real sequence  supported on $  {N} <   n  \leqslant    {2N}  $.   
Define
\begin{align*}
	\| \mathcal{A} \|^{}   = \bigg(  \sum_{    n \sim    {N} } |a_{n}|^2 \bigg)^{1/2}. 
\end{align*}
Subsequently, we shall deal with  the smoothed spectral averages as in \eqref{1eq: S, cusp}	 and \eqref{1eq: E, Eis}:
\begin{align*}
	S (\mathcal{A})    =	   \sum_{j = 1 }^{\infty}   \omega_j { h  ( t_j ) }   \bigg| \sum_{ n  }  a_{n} \lambda_{j} (n) n^{i t_j}  \bigg|^2   ,   \qquad 
T (\mathcal{A})    =	   \frac 1 {\pi} 	\int_{-\infty}^{\infty}   \omega (t) {h (t)}  \bigg| \sum_{ n }  a_{n} \sigma_{2 i t} (n)   \bigg|^2  \nd t ,  
\end{align*}  
in which $h (t)$ is the spectral weight function defined as in \eqref{1eq: weight k}  or \eqref{3eq: defn h(t)}.

\subsection{Application of the Kuznetsov Trace Formula}

Choose the spectral weight in Lemma \ref{lem: Kuznetsov}  to be $h (t; \sqrt{m/n})$ as defined in \eqref{3eq: defn h(t)} and \eqref{3eq: defn h(t; y)}. Then multiply both sides of \eqref{2eq: Kuznetsov}  by ${a}_{m}   {a}_{n} $, and  sum over $  m, n \sim  {N} $.  Note that $$\mathrm{Re} \big( (m/n)^{i t} \big) = \cos \big(2t \log \sqrt{m/n} \big), $$
so the Kuznetsov formula \eqref{2eq: Kuznetsov}   in Lemma \ref{lem: Kuznetsov}  yields
\begin{align}\label{4eq; C = D + P}
	S (\mathcal{A})  + T (\mathcal{A})  = D (\mathcal{A}) + P (\mathcal{A}), 
\end{align}
with diagonal 
\begin{align}
	D (\mathcal{A}) = H \cdot \sum_{  n } |a_{n}|^2 , \qquad H   = \frac{1}{\pi^2}\displaystyle\int_{-\infty}^{\infty}h(t) \tanh(\pi t)t \nd t , 
\end{align} 
and off-diagonal 
\begin{align}
	P (\mathcal{A}) =  \sum_{c =1}^{\infty}   	\mathop{\sum\sum}_{m, n  }   {a}_{m}   {a}_{n} \frac{S (m, n;c)} {c} H \bigg( \frac {4\pi \sqrt{m n}} {c} , \sqrt{\frac {m} {n} } \bigg) . 
\end{align}
 \subsection{Proof of Theorem \ref{thm: asymptotic large sieve}} By a simple evaluation of $ H$, we have 
 \begin{align}\label{4eq: bound D(a)}
 	D (\mathcal{A}) =  \frac {2} {\pi \sqrt{\pi}} M T  (1 + O  (T^{-A}  )  ) \|\mathcal{A}\|^2 . 
 \end{align} By Lemmas \ref{lem: x<1}, \ref{lem: x>1}, and \ref{lem: analysis of I}, it follows that the Bessel $H$-integral may be transformed into $I$-integral by \eqref{4eq: H = I} while the $c$-sum may be truncated effectively at $ c \asymp N/ T $:
\begin{align}\label{4eq: P(a)} 
P (\mathcal{A})   =   \mathrm{Re}   \sum_{c \shskip \Lt N/T  }    \mathop{\sum   \sum}_{m, n  }    {a}_{m}   {a}_{n}   \frac {S (m, n;c)} { c } e \Big(     \frac {m    +    n} {c}    \Big)  I \Big(   \frac {\pi m } {c},   \frac {\pi n } {c} \Big)   +   O  \bigg(   \frac {N^{3/2+\vepsilon}} {T^{A}} \| \mathcal{A} \|^{2}   \bigg)   ,
\end{align}  
for any $A \geqslant 0$;  the error  is estimated trivially by \eqref{2eq: quad form, Kloosterman}.   By the identity \eqref{2eq: S = V}, we have
\begin{align}\label{4eq: P(a), 2} 
	P (\mathcal{A})  = \mathrm{Re}   \mathop{\sum\sum}_{c q \Lt N/T}      \mathop{\sum\sum}_{m, n  }   {a}_{m}   {a}_{n} \frac {V_{q} (m, n; c)} {c q}    I \Big(  \frac {\pi m } {cq},   \frac {\pi n} {cq} \Big) + O  \bigg( \frac {N^{3/2+\vepsilon}} {T^{A}} \| \mathcal{A} \|^{2} \bigg) .
\end{align}

\begin{lem}\label{lem: P(A)}
Let $P_{\natural} (\mathcal{A})  $ denote the quadruple sum in {\rm\eqref{4eq: P(a), 2}}. Then
	\begin{align}\label{4eq: bound Pn(a)}
		P_{\natural} (\mathcal{A}) \Lt M T \sum_{q \Lt N/T}  \frac 1 {q} \int_{-M^{\vepsilon}/ M}^{M^{\vepsilon}/ M} \sum_{c \Lt N/ T q} \frac 1 {c}  \, \sumx_{   \valpha      (\mathrm{mod} \, c) } \bigg|   \sum_{   n   }  a_{n}   e \Big(   \frac {\widebar{\valpha}     n} {c} \Big)  e \bigg( \frac {n t} {c q} \bigg) \bigg|^2  \nd t . 
	\end{align} 
\end{lem}

\begin{proof}
	By the  integral expression  in  \eqref{4eq: I}, $P_{\natural} (\mathcal{A}) $ is expanded into 
\begin{align*} 
\begin{aligned}
	P_{\natural} (\mathcal{A}) =	    M T     \int_{-M^{\vepsilon}/ M}^{M^{\vepsilon}/ M}  g (r)    \mathop{\sum\sum}_{c q \Lt N/T}  \mathop{\sum    \sum}_{m, n  }   {a}_{m}   {a}_{n}     \frac {V_{q} (m,    n; c)} {c q} e   \Big( \frac {m} {c q} \uprho_{+}  (r)     -     \frac {n} {c q} \uprho_{-}  (r)   \Big)   \nd r . 
\end{aligned}
\end{align*}
Next, we insert the definition of $V_{q} (m, n; c)$ as in \eqref{2eq: defn V} to split the $m$- and $n$-sums, so 
\begin{align*}
	P_{\natural} (\mathcal{A}) =  M T  \int_{-M^{\vepsilon}/ M}^{M^{\vepsilon}/ M}   g (r)   \mathop{\sum\sum}_{c q \Lt N/T} \frac 1 {cq}   \bigg(   \mathop{\sum_{\valpha     (\mathrm{mod}\, c)}}_{ (\valpha     (q-\valpha    ), c) = 1 }    I_{\valpha    }^{+} (r; c, q; {\mathcal{A}})    {I_{q - \valpha    }^{-} (r; c, q;  {\mathcal{A}} )}   \bigg) \nd r , 
\end{align*}
in which
\begin{align*} 
	I_{\valpha    }^{\pm} (r; c, q; \mathcal{A}) = \sum_{n} a_n   e \Big(    \frac {  \widebar{\valpha    } n } {c} \Big) e \Big(  \hspace{-2pt} \pm \frac {n} {c q} \uprho_{\pm} (r) \Big). 
\end{align*}
Now we bound $P_{\natural} (\mathcal{A})$ in the trivial manner and apply the AM--GM inequality to the inner $I$-product, so that the $\valpha$-sum splits into 
\begin{align*}
 \mathop{\sum_{\valpha     (\mathrm{mod}\, c)}}_{ (\valpha     (q-\valpha    ), c) = 1 }  \big|  I_{\valpha    }^{+} (r; c, q; {\mathcal{A}})  \big|^2 +  \mathop{\sum_{\valpha     (\mathrm{mod}\, c)}}_{ (\valpha     (q-\valpha    ), c) = 1 } \big|  {I_{ q - \valpha    }^{-} (r; c, q;  {\mathcal{A}} )}\big|^2 .
\end{align*}
By the change $ q - \valpha  \ra \valpha$ in the second  sum and then the omission of the coprimality condition $(q-\valpha, c) = 1$ in both sums, this is further bounded by 
\begin{align*}
	 {\sumx_{\valpha     (\mathrm{mod}\, c)}}  \lp  \big|  I_{\valpha    }^{+} (r; c, q; {\mathcal{A}})  \big|^2 +   \big|  {I_{  \valpha    }^{-} (r; c, q;  {\mathcal{A}} )}\big|^2 \rp .
\end{align*} 
It follows that 
\begin{align*} 
	P_{\natural} (\mathcal{A})  \Lt M T \sum_{q \Lt N/T}  \frac 1 {q}  \sum_{\pm}  \int_{-M^{\vepsilon}/ M}^{M^{\vepsilon}/ M}     \sum_{c   \Lt N/  T q} \frac 1 {c}   \,     {\sumx_{\valpha     (\mathrm{mod}\, c)}}    \bigg|  \sum_{n} a_n   e \Big(   \frac {  \widebar{\valpha    } n } {c} \Big) e \Big( \hspace{-2pt} \pm \frac {n} {c q} \uprho_{\pm} (r) \Big) \bigg|^2 \nd r . 
\end{align*}
Finally, since $$ \uprho_{\pm}' (r) = \cosh r \pm \sinh r = 1 + O (M^{\vepsilon}/ M) $$ on the integral domain, the change of variable $t = \pm \uprho_{\pm} (r)$ yields \eqref{4eq: bound Pn(a)}; here we may enlarge the resulting integral domain slightly by positivity and adjust $\vepsilon$ by our $\vepsilon$-convention. 
\end{proof}

By combining \eqref{4eq; C = D + P},  \eqref{4eq: bound D(a)}, \eqref{4eq: P(a), 2}, and \eqref{4eq: bound Pn(a)}, we obtain   Theorem \ref{thm: asymptotic large sieve} (with slight abuse of notation, the negligible error from $P (\mathcal{A})$ has been absorbed into $D (\mathcal{A})$). 

\subsection{Proof of Corollary \ref{cor: large sieve}}   
By   the large sieve of Young  as in Lemma  \ref{lem: Young's LS} with $\gamma = 1$,  $\tau = M^{\vepsilon}/ M$,  $v = q$, and $C = O (N / q T)$  to the expression in \eqref{4eq: bound Pn(a)}, we have  
\begin{align*}
	P_{\natural} (\mathcal{A}) \Lt   M T \sum_{q \Lt N/T}  \frac 1 {q}  \bigg(  \frac {M^{\vepsilon}} {M} \frac {N} {q T}  + q \log N \bigg)     \|  \mathcal{A} \|^{2} \Lt   M N   (TN)^{\vepsilon} \|  \mathcal{A} \|^{2},  
\end{align*}
and, in view of  \eqref{4eq: P(a), 2}, 
\begin{align}\label{4eq: bound P(a)}
	P (\mathcal{A}) \Lt \lp  M N + \frac {N^{3/2}} {T^{A}}  \rp (TN)^{\vepsilon}  \| \mathcal{A} \|^{2} . 
\end{align}

It follows from \eqref{4eq; C = D + P}, \eqref{4eq: bound D(a)}, and \eqref{4eq: bound P(a)} that
\begin{align*}
	S (\mathcal{A})  + T (\mathcal{A})  \Lt \lp M T + M N + \frac {N^{3/2}} {T^{A}}  \rp (TN)^{\vepsilon}  \| \mathcal{A} \|^{2}, 
\end{align*}
and hence
\begin{align*} 
	 	\sum_{ T - M \leqslant t_j \leqslant T+M }  \omega_j  \bigg| \sum_{     n  }  a_{n} \lambda_{j} (n) n^{it_j}  \bigg|^2  \Lt    	\lp M T + M N + \frac {N^{3/2}} {T^{A}}  \rp (TN)^{\vepsilon}    \sum_{   n   } |a_{n}|^2 ;
\end{align*}
its validity may be easily extended to $1 \leqslant M \leqslant T$. 
 Finally, to finish the proof, we enlarge $M \ra M + N^{\vepsilon}$ and $T \ra T + N^{\vepsilon}$ simultaneously,  and choose $A = 3/2\vepsilon$ to absorb $N^{3/2}/ T^{A}$ into $M T$. 

\section{Mean Lindel\"of Hypothesis: Proof of Theorem \ref{thm G:(3)}}

For $h (t)$ as in \eqref{1eq: weight k}  or  \eqref{3eq: defn h(t)}, consider 
\begin{align}
	\label{5eq: S(T, M)}  S_{\phi}   & =  \sum_{j = 1 }^{\infty}   \omega_j { h  ( t_j ) }   \big| L (1/2+it_j, \phi \times u_j)  \big|^2,  \\
	\label{5eq: E (T, M)}   T_{\phi}   &= \frac 1 {\pi} 	\int_{-\infty}^{\infty}   \omega (t) {h (t)}  \big|L (1/2+2it , \phi) L (1/2 , \phi)  \big|^2  \nd t. 
\end{align}
Our aim is to prove 
\begin{align} 
	\label{5eq: main bound} S_{\phi} + T_{\phi} \Lt_{\vepsilon, \phi}  M T^{1+\vepsilon} + \frac {T^{5/2+\vepsilon}} {M^2} , 
\end{align}
so that \eqref{1eq: mean Lindelof, GL(3)}    follows for all $ \sqrt{T} \leqslant M \leqslant T$ as in Theorem \ref{thm G:(3)}.

\subsection{Sketch} \label{sec: sketch}

For simplicity, in this sketch of proof, we shall omit the factor $ T^{\vepsilon}$. 

Essentially, we need to work with the sum
\begin{align*}
S(N) + T(N) = \hspace{-1pt}   \sum    \omega_j { h  ( t_j ) }   \bigg| \hspace{-1pt}  \sum_{n\sim N}  \frac {A(1        , n) \lambda_j (n)}{n^{1/2 + i t_j} }    \bigg|^2 \hspace{-2pt}  + \frac 1 {\pi} \hspace{-1pt} 	\int \hspace{-1pt}   \omega (t) {h (t)}  \bigg| \hspace{-1pt}  \sum_{n \sim N} \frac {{A(1        , n)} \sigma_{2 i t} (n)}  {\sqrt{n}}     \bigg|^2 \nd t,
\end{align*}
for the length $N \approx T^{3/2  }$. By Theorem \ref{thm: asymptotic large sieve} and Remark \ref{rem: real assump}, we arrive at the diagonal
\begin{align*}
	\breve{D} (N) = M T \sum_{n \sim N} \frac {|A(1, n)|^2} {n}  \Lt     M T , 
\end{align*}
by the Rankin--Selberg estimate  \eqref{2eq: RS}, and the off-diagonal
\begin{align*}
\breve{P} (N)	=   {M T}   \sum_{q \Lt N/ T}    \int_{-1/ M}^{1/ M}  \sum_{c \Lt N/ T q} \frac 1 {c q}  \, \sumx_{   \valpha      (\mathrm{mod} \, c) } \bigg|  \sum_{ n \sim N }  \frac{A(1, n)} {\sqrt{n} }  e \Big(   \frac {\widebar{\valpha}     n} {c} \Big)  e \bigg( \frac {n t} {c q} \bigg)  \bigg|^2   \nd t . 
\end{align*}

The Vorono\"i summation formula in Lemma \ref{lem: Voronoi} and stationary phase analysis in Lemma \ref{lem: analysis of integral} yield roughly the dual expression:
\begin{align*}
	\frac{  T  } { \sqrt{M}  }   \sum_{q \shskip \Lt N^{2/3} / M }    \int_{-\sqrt{M}}^{ \sqrt{M} }   \sum_{c \Lt N / T q}   \frac {1} {c^2 q}      \sum_{   \valpha      (\mathrm{mod} \, c) }       \bigg|  \sum_{n \shskip \asymp  {N^2  } / { M^3 q^3}}   \frac { A(n , 1) }  {\sqrt{n}}    {S (   n,  {\valpha}   ;c   )}   e \bigg(    \frac {2   \sqrt{qn} t} {c  }  \bigg)   \bigg|^2     { \nd t }.  
\end{align*}
Note here that in order for the $n$-sum not to be void the $q$-sum is forced to be shortened into $ q \Lt N^{2/3} / M $ (actually, it will be those small $q$ that make the main contribution) and the coprimality condition $(\valpha, c) = 1$ is dropped after the Vorono\"i summation formula (by non-negativity). 

Next, we open the square and calculate the exponential sum to transform the expression above into 
\begin{align*}
	\frac{  T  } { \sqrt{M}  }  \sum_{q \shskip \Lt N^{2/3} / M }    \int_{-\sqrt{M}}^{ \sqrt{M} }   \sum_{c \Lt N / T q}   \frac {1} {c q} \,     \sumx_{   \beta      (\mathrm{mod} \, c) }       \bigg|  \sum_{n \shskip \asymp  {N^2  } / { M^3 q^3}}   \frac { A(n , 1) } {\sqrt{n}}   e \bigg(   \frac {\widebar{\beta}     n} {c} \bigg)  e \bigg(    \frac {2   \sqrt{qn} t} {c  }  \bigg)   \bigg|^2     { \nd t }.  
\end{align*}

Finally, by an application of Young's hybrid large sieve \eqref{2eq: hybrid ls, Young, 2} and the Rankin--Selberg estimate \eqref{2eq: RS}, we have the bound 
\begin{align*}
\breve{P} (N) \Lt	\frac {T} {\sqrt{M}} \sum_{ q \shskip \Lt N^{2/3} / M } \frac 1 {q} \bigg(\sqrt{M} \frac {N} {T q} + \frac {N} {M^{3/2} q^2}  \bigg) \Lt N + \frac {T N} {M^2} \Lt T^{3/2} + \frac {T^{5/2}} {M^2} ,
\end{align*} 
and in conclusion 
\begin{align*}
	S(N) + T(N)  \Lt M T + T^{3/2} + \frac {T^{5/2}} {M^2} \Lt M T +  \frac {T^{5/2}} {M^2}. 
\end{align*}

\subsection{Initial Reductions}  
Let us assume  $ |t_j - T | \leqslant M^{1+\vepsilon}$ as otherwise $h (t_j)$ is exponentially small.      
By \cite[Theorem 5.3]{IK}, it follows from \eqref{2eq: defn L(s)}--\eqref{2eq: FE, GL(3) x GL(2)} the approximate functional equation:  
\begin{align*}
	\begin{aligned}
		L  (  1 / 2 + i t_j , \phi \times u_j  )      = &   \mathop{\sum \sum}_{ n_1  , n     }  \frac{A(n_1,n    )\lambda_j (n    )}{(n_1^2 n    )^{1/2+it_j} }   V_{\delta_j}    \big( n_1^2 n    ; 1/2+it_j \big)   \\
		& +  \epsilon_j (\phi)    \mathop{\sum \sum}_{ n_1, n     }  \frac{\overline{A(n_1,n    )}\lambda_j (n    )}{(n_1^2 n    )^{1/2-it_j} }   \widetilde{V}_{\delta_j}    \big( n_1^2 n    ; 1/2-it_j \big) ,  
	\end{aligned}
\end{align*}
where $ \epsilon_j (\phi) $ has unity norm,
\begin{align*}
	V_{\delta} (y; 1/2+it) =  \frac 1 {2   \pi i }   \int_{ (3)}   & G_{\delta}  (v, it;  \phi )   y^{ - v}   \frac {\nd v} {v}, 
\end{align*}
\begin{align*}
	G_{\delta}  (v, it; \phi   ) =  
	\frac { \gamma (1/2 + 2i t + \delta   + v, \phi  )} {  \gamma (1/2 + 2i t  + \delta  , \phi  )} \frac {  \gamma (1/2 + \delta    + v, \phi  )} {  \gamma (1/2 + \delta    , \phi  )} \exp({v^2}) ,
\end{align*}  
with $\gamma (s, \phi)$ defined as in \eqref{2eq: gamma (s, phi)} (see also \eqref{2eq: gamma (s, phi uj)}), while  $ \widetilde{V}_{\delta }    ( y ; 1/2-it) $ is similarly defined with $\phi \ra \widetilde{\phi}$.  By \cite[Proposition 5.4]{IK}, one may effectively restrict the sums above to the range  $ n_1^2 n     \leqslant T^{3/2+\vepsilon} $ at the cost of a negligibly small error. 
In order to facilitate our analysis, we use  the following expression  due to Blomer \cite[Lemma 1]{Blomer}  (slightly modified): 
\begin{align*}
	V_{\delta} (y; 1/2+it)  =  \frac 1 {2   \pi i }   \int_{ \vepsilon - i U}^{\vepsilon + i U}   G_{\delta}  (v, it;  \phi )     y^{ - v}   \frac {\nd v} {v} + {O_{\vepsilon}} \bigg( \frac { T^{  \vepsilon} } {y^{ \vepsilon} \exp ({U^2/2}) } \bigg) . 
\end{align*}
The error term above is negligibly small if we choose $ U = \log T$. Note that  for any $v$ on the integral contour, 
\begin{align*}
	G_{\delta}  (v, it; \phi   )  = O_{\vepsilon, \phi}  ( T^{\vepsilon} )  , 
\end{align*} 
by the Stirling  formula, provided that $ ||t| - T|\leqslant M^{1+\vepsilon}$.  

By a smooth dyadic partition and the Cauchy--Schwarz  inequality, we infer that up to a negligible error 
\begin{align}\label{2eq: AFE} 
	| L  (  1 / 2 + i t_j , \phi \times u_j  )   |^2 \Lt  {T^{\vepsilon}}     \max_{P \shskip \leqslant  T^{3/2+\vepsilon}}   \int_{    \vepsilon - i  \log T}^{\vepsilon + i \log T}   \big| S_j^v (P)  \big|^2     \nd v ,
\end{align}
where   $P$ are dyadic in the form $ 2^{k/2}  $ ($k \geqslant -1$),  and  
\begin{align}
	S_j^v (P)  = \frac 1 {\sqrt{P}} \mathop{\sum\sum}_{ n_1,  n     }  \frac{A(n_1,n    )\lambda_j (n    )} {(n_1^2 n    )^{ i t_j} } \varww_{ v} \bigg(\frac {n_1^2 n    } {P} \bigg), \quad \ \	\varww_{ v} (x) =   \frac {\varvv (x)}   {x^{1/2+ v} }  , 
\end{align} 
for a certain fixed $\varvv \in C_c^{\infty} [1 , {2}]$. Note that $\varww_{ v} (x)$ is $\log T$-inert according to Definition \ref{defn: inert}; namely $ \varww_v^{(i)} (x) \Lt_{\vepsilon, i} \log^i T  $. Further, by the Cauchy--Schwarz inequality, we have
\begin{align}\label{5eq: S (N, uj)}
	\big|S_j^v (P)\big|^2  \Lt T^{\vepsilon}  \sum_{n_1 \Lt \sqrt{  P}} \frac 1 {n_1} \bigg|\frac {n_1} {\sqrt{P}} \sum_{n}  {A(n_1    , n)\lambda_j (n)}{n^{- i t_j} } \varww_{ v} \bigg(\frac {n} {P /n_1^2}  \bigg)\bigg|^2.  
\end{align} 

Moreover,  if $ \lambda_j (n) $ were replaced by $ n^{ it} {\sigma_{-2 i t} (n)}$, then the arguments above apply  in parallel to the (Eisenstein) case of $  |L (1/2+2it , \phi) L (1/2 , \phi)   |^2 $. Similar to \eqref{5eq: S (N, uj)}, the final expression that we need to consider reads:   
\begin{align}
\sum_{n_1 \Lt \sqrt{ P}} \frac 1 {n_1} \bigg|	 \frac {n_1} {\sqrt{P}} \sum_{n}   A(n_1    , n) \sigma_{-2 i t} (n)  \varww_{ v} \bigg(\frac {n} {P/n_1^2}  \bigg) \bigg|^2. 
\end{align}

\delete{Moreover,  the arguments above apply  in parallel for the Eisenstein case, with $ \lambda_j (n) $ replaced by $ n^{it} \sigma_{2 i t} (n)$. For instance, we have 
	\begin{align}
		\begin{aligned}
			L (1/2+2it, \phi) L (1/2, \phi) &  =   \mathop{\sum \sum}_{ n_1, n     }  \frac{A(n_1,n    ) \sigma_{2it} (n    )}{n_1^{1+2it} n    ^{1/2}  }   V_{0}    ( n_1^2 n    ; 1/2+it)   \\
			& +  \epsilon (t, \phi)    \mathop{\sum \sum}_{ n_1, n     }  \frac{\overline{A(n_1,n    )} \sigma_{-2it} (n    )}{ n_1^{1-2it} n    ^{1/2} }   \widetilde{V}_{0}    ( n_1^2 n    ; 1/2-it) ,   
		\end{aligned}
\end{align}}


\begin{lem}\label{lem: S + T}
	  For $  N  n_1^2 \Lt T^{3/2+\vepsilon} $,  define
	  \begin{align}
	  	S  (n_1; N)  & = \sum_{j = 1 }^{\infty}   \omega_j { h  ( t_j ) }   \bigg|\frac { 1    } {\sqrt{N}} \sum_{n}  {A(n_1    , n)\lambda_j (n)}{n^{- i t_j} } \varww  \Big(\frac {n} {N}  \Big) \bigg|^2, \\
	  	T  (n_1; N)  & =  \frac 1 {\pi} 	\int_{-\infty}^{\infty}   \omega (t) {h (t)}  \bigg|\frac { 1    } {\sqrt{N}} \sum_{n}  {A(n_1    , n) \sigma_{-2 i t} (n)}  \varww  \Big(\frac {n} {N}  \Big) \bigg|^2 \nd t , 
	  \end{align}
  where $ \varww \in C_c^{\infty} [1, 2]$ is $\log T$-inert in the sense of Definition {\rm\ref{defn: inert}}.  
  Then 
  \begin{align}\label{5eq: bound for S+T}
 S  (n_1; N) + T  (n_1; N) \Lt   \bigg(1 + \frac {M T   } {N}\bigg) T^{\vepsilon}   \sum_{ n \sim N}   |A (n_1, n)|^2   + \bigg( n_1 + \frac {T} {M^2  }  \bigg)   {N n_1 T^{ \vepsilon}} . 
  \end{align}
\end{lem}

By the discussion above, the estimate in \eqref{5eq: main bound} readily follows from Lemma \ref{lem: S + T}. Given \eqref{5eq: bound for S+T}, for $P \leqslant T^{3/2+\vepsilon}$, by the Rankin--Selberg estimate \eqref{2eq: RS, 2}, we have
\begin{align*}
	 \sum_{ n_1 \Lt \sqrt{P} } \frac {(S+T) (n_1; P/n_1^2) } {n_1}  \Lt T^{\vepsilon}  \bigg(M T + T^{3/2} + \frac {T^{5/2} } {M^2} \bigg) \Lt T^{\vepsilon} \bigg(M T +   \frac {T^{5/2} } {M^2} \bigg) . 
\end{align*} 

\subsection{Application of Theorem \ref{thm: asymptotic large sieve}} 

Set  
\begin{align*}
	\overline{a}_n  = \frac {1 } {\sqrt{N}}  A(n_1, n) \varww  \Big(\frac {n} {N}  \Big),
\end{align*} 
so that Theorem \ref{thm: asymptotic large sieve} and Remark \ref{rem: real assump} yield 
\begin{align}
	 S  (n_1; N) + T  (n_1; N) \Lt \breve{D}  (n_1; N) + \breve{P}  (n_1; N) , 
\end{align}
 where 
 \begin{align}
 	\breve{D}  (n_1; N) = \frac {M T   } {N}   \sum_{  n \sim  N} |A (n_1, n)|^2  , 
 \end{align}
\begin{align} \label{5eq: P(N)}
		\breve{P}  (n_1; N)  = \frac{M T} {N}  \sum_{q \Lt N/T}  \frac 1 {q} \int_{-M^{\vepsilon}/ M}^{M^{\vepsilon}/ M}  \sum_{c \Lt N/ T q} \frac 1 {c}  \, \sumx_{   \valpha      (\mathrm{mod} \, c) } \big|  P_{\valpha} (t/q;  c, n_1; N)  \big|^2   \nd t ,  
\end{align} 
with
\begin{align}\label{5eq: P (t/q)}
P_{\valpha} (t/q;  c, n_1; N)  =	 \sum_{ n }  A(n_1, n)   e \Big(   \frac {\widebar{\valpha}     n} {c} \Big)  e \bigg( \frac {n t} {c q} \bigg) \varww  \Big(\frac {n} {N}  \Big) . 
\end{align}
Further, in order to facilitate our analysis, we truncate  the $t$-integral at $ |t| = 1/ M N $ say and the $q$-sum at $ q = T^{\vepsilon}$, and then apply a dyadic partition for $ 1/MN \leqslant |t| \leqslant M^{\vepsilon} / M $; the resulting error is satisfactory: 
 \begin{align*}
 	 O \bigg( \frac {N + MT} {N} T^{\vepsilon} \sum_{  n \sim N} |A (n_1, n)|^2    \bigg), 
 \end{align*} 
by  trivial estimation or by Young's hybrid large sieve in Lemma \ref{lem: Young's LS}.  \delete{Accordingly, we write
\begin{align}
	P (n_1; N) \Lt \sum_{1/M N \Lt \tau \Lt M^{\vepsilon} / M }   P^{\tau} (n_1; N) + \frac {MT + N} {N} T^{\vepsilon} \sum_{  n \sim N} |A (m, n)|^2, 
\end{align}
where $ \tau  $ is dyadic in the form $1 / 2^{k}$. }

\subsection{Application of the Vorono\"i Summation Formula}  
Subsequently,  let $q > T^{\vepsilon}$, $|t| \sim \tau$ for  $ \tau \Lt M^{\vepsilon} / M$.  By applying the Vorono\"i summation formula in Lemma \ref{lem: Voronoi}, the sum $P_{\valpha} (t/q;  c, n_1; N) $ in \eqref{5eq: P (t/q)} is transformed into 
\begin{align}\label{5eq: after Voronoi}
P_{\valpha} (t/q;  c, n_1; N) =	   \sum_{\pm}     \sum_{d\mid c n_1   }    d      \sum_{n  }   A(n , d)    \frac {S\left(\pm n,  {\valpha}n_1;c n_1   /d\right)} {c^2 n_1   }   \Omega_N \bigg( \hspace{-2pt} \mp \frac{d^2 n} {c^3 n_1   }; \frac {t} {cq} \bigg)  , 
\end{align}
where according to \eqref{2eq: Hankel}
\begin{align}
	\Omega_N ( y; r )  = \int J_{\phi} ( -  x y) e  (   x r  ) \varww  \Big(\frac {x} {N}  \Big) \nd x  . 
\end{align} 

\subsection{Analysis for the Hankel Transform} 
For  $  | N y| \Gt T^{\vepsilon}$, it is permissible to use the asymptotic expansion for $J_{\phi} ( - x y)$ with a negligibly small error term (choose $K = \lfloor 3 A /\vepsilon \rfloor$ + 1, say) as in \eqref{4eq: asymptotic, Bessel, R}.  It follows that 
\begin{align}
	\Omega_N ( y; r ) = \frac 1 {\sqrt[3]{  y}}  \int   e \big( x r -  {3 \sqrt[3]{x y} }  \big) \varww_{\phi}  \Big(\frac {x} {N}  \Big) \frac {\nd x } { \sqrt[3]{x} } + O  (T^{-A} ),
\end{align}
for some $\log T$-inert function  $ \varww_{\phi}  \in C_c^{\infty} [1, 2] $.

For $y, r > 0$, we make the change $  x  \ra x \sqrt{y/ r^3}  $ so that 
\begin{align*} 
	\Omega_N (  y; r ) = \frac 1 {r}  \int   e \big( \sqrt{y/ r}  (x -  {3 \sqrt[3]{x } } )  \big) \varww_{\phi}  \bigg(\frac {x   } {N \sqrt{r^3 / y}}  \bigg) \frac {\nd x } { \sqrt[3]{x} } + O  (T^{-A} ). 
\end{align*}
By applying Lemma \ref{lem: analysis of integral} with $\lambda = \sqrt{y/ r}$, $ \rho = N \sqrt{r^3 / y} $, and $X = \log T$, we infer that $\Omega_N (  y; r )$ is negligibly small unless $ y \asymp N^2 r^3 $, in which case  
\begin{align}\label{5eq: Omega, +}
	 \Omega_N (  y; r ) = \frac{e  (   - 2 \sqrt{y/r}  ) \varvv_{-}  (  y, r  )  } { \sqrt{ N r^3} } + O  (T^{-A} ),
\end{align} 
where  the function $ \varvv_{-}  ( y, r  ) $ is $\log T$-inert.  
Also note here that $ \lambda \sqrt[3]{\rho} =  \sqrt[3]{N y} \Gt T^{\vepsilon} $ and $ \sqrt{\lambda}  \asymp \sqrt{N r}$ for $y \asymp N^2 r^3$.  Similarly, for $y, r < 0$, we have
\begin{align} \label{5eq: Omega, -}
	\Omega_N (  y; r ) = \frac{e  (   2 \sqrt{y/r}  ) \varvv_{+}  (  y, r  )  } { \sqrt{ N r^3} } + O  (T^{-A} ). 
\end{align} 
Moreover, in the case $ y r < 0 $, the integral $  \Omega_N (  y; r ) $ is always negligibly small. 

Let us return to the setting above as in \eqref{5eq: P(N)} and \eqref{5eq: after Voronoi}. First of all, as $c \Lt N/ T q  $ and $ N n_1^2 \Lt T^{3/2+\vepsilon}$,   by our assumption $q > T^{\vepsilon}$, for $|y| =    d^2 n/ c^3 n_1$,  we have indeed
\begin{align*}
	| N y |   \Gt \frac {N} {c^3 n_1} \Gt \frac {T^3 q^3} { N^2 n_1 } \Gt \frac {q^3 n_1^3} {T^{\vepsilon}} \Gt T^{\vepsilon}. 
\end{align*}  
Also note that the condition $ y \asymp N^2 r^3 $ amounts to 
\begin{align*}
	n \asymp \frac {N^2 n_1 \tau^3} {d^2 q^3}, 
\end{align*} 
the sign in \eqref{5eq: after Voronoi} must be opposite to that of $t$,  
while 
\begin{align*}
	\sqrt{\frac {y} {r}} = \frac {d \sqrt{qn}} {c \sqrt{n_1 t}}, \qquad \frac 1 {  {N |r|^3 } } = \frac { c^3 q^3 } { {N} |t|^3}     . 
\end{align*} 
After the applications of the Vorono\"i summation formula and the stationary phase analysis, we no longer need the restrictions $(\valpha, c) = 1$ and $q > T^{\vepsilon}$. Consequently, in view of \eqref{5eq: P(N)}, \eqref{5eq: after Voronoi}, \eqref{5eq: Omega, +}, and \eqref{5eq: Omega, -}, up to a negligible error, we have
\begin{equation*}
\begin{split}
 	\frac{M T} {N^2 n_1^2} &  \sum_{\pm}    \sum_{q \Lt N/T} q^2    \int_{\,\tau}^{2 \tau}     \sum_{c \Lt N/ T q}   \frac 1 {c^2}    \sum_{   \valpha      (\mathrm{mod} \, c) }       \\
	\cdot \, &  \bigg|    \sum_{d\mid c n_1   }   d   \sum_{n \shskip \asymp  {N^2 n_1 \tau^3} / {d^2 q^3}}     A(n , d)      {S\left(\pm n,  {\valpha}n_1;c n_1/d\right)}   e \bigg( \hspace{-2pt} \pm \frac {2 d \sqrt{qn}} {c \sqrt{n_1 t}}  \bigg) \varvv_{\pm} \bigg(\frac{d^2 n} {c^3 n_1   }, \frac {t} {cq} \bigg)  \bigg|^2   \frac {\nd t } {\,t^3 } . 
\end{split}
\end{equation*}
For notational succinctness, let us only consider the $+$ contribution. 
Next, we  make the change   $1/\sqrt{t} \ra t$ and pull the $d$-sum out of the square by the Cauchy inequality, giving
\begin{equation*}
	\begin{split}
	& 	\frac{M T^{1+\vepsilon}  } {N^2 n_1^2 \tau^{3/2}}    \sum_{q \Lt N/T} q^2   \int_{1/ \sqrt{2\tau}}^{1/\sqrt{\tau} }   \sum_{c \Lt N/ T q}   \frac 1 {c^2} \sum_{d\mid c n_1   }   d^2      \\
		  	\cdot &  \sum_{   \valpha      (\mathrm{mod} \, c) }       \bigg|  \sum_{n \shskip \asymp  {N^2 n_1 \tau^3} / {d^2 q^3}}     A(n , d)      {S\left(   n,  {\valpha}n_1;c n_1/d\right)}   e \bigg(    \frac {2 d \sqrt{qn} t} {c \sqrt{n_1  }}  \bigg) \varvv_{+} \bigg(\frac{d^2 n} {c^3 n_1   }, \frac {1} {cq t^2} \bigg)  \bigg|^2     { \nd t }   . 
	\end{split}
\end{equation*} 

\subsection{Evaluation of the Exponential Sum}
After opening the square in the $\valpha$-sum,  we obtain the exponential sum
\begin{align*}
	\sum_{   \valpha      (\mathrm{mod} \, c) } &  S (  m , \valpha n_1; c n_1/d) S (  n, \valpha n_1; c n_1/d) \\
	& = \sum_{   \valpha      (\mathrm{mod} \, c) } \mathop{\sumx \sumx}_{ \beta, \gamma (\mathrm{mod}\, c n_1/d)} e \bigg(   \frac { \widebar{\beta} m  - \widebar{\gamma} n  } {c n_1/d} + \frac {\valpha d (\beta - \gamma)} {c}  \bigg). 
\end{align*}
By orthogonality, the $\valpha$-sum yields the congruence condition $ d (\beta - \gamma) \equiv 0 \,(\mathrm{mod}\, c) $, or equivalently $ \beta \equiv \gamma \,(\mathrm{mod}\, c/ (c, d)) $. For brevity, set 
$$   c' = \frac {c} {(c, d)}, \qquad n_1' = \frac {n_1} {d / (c, d)}.   $$  
Thus we may write $ \gamma = \beta + c' \vnu $ for $ \vnu\, (\mathrm{mod} \, n_1')$ such that $( \beta + c' \vnu,  n_1') = 1$, so   the whole $\valpha$-sum is turned into
\begin{equation*}
	\begin{split}
	  c  \,  \sumx_{\beta (\mathrm{mod}\, c' n_1')}    \mathop{\sum_{\vnu (\mathrm{mod} \, n_1')} }_{(\beta+c'\vnu, n_1') = 1 }    S_{\beta}^{+}  (\sqrt{q} t; c, n_1,d; N)  \overline{S_{\beta+c' \vnu}^{+} (\sqrt{q} t; c, n_1,d; N) }     ,
	\end{split}
\end{equation*} 
where 
\begin{align*}  
	S^{+}_{\beta}  (\sqrt{q} t; c, n_1,d; N)   =   \sum_{n \shskip \asymp  {N^2 n_1 \tau^3} / {d^2 q^3}}     A(n , d)   e \bigg(   \frac {\widebar{\beta} n } {c n_1/d} \bigg)  e \bigg(      \frac {2 d \sqrt{qn} t} {c \sqrt{n_1  }}  \bigg) \varvv_{+} \bigg(\frac{d^2 n} {c^3 n_1   }, \frac {1} {cq t^2} \bigg). 
\end{align*}
Similar to the proof of Lemma \ref{lem: P(A)}, we apply  the AM--GM inequality to the   $S$-product  so that the double sum above is  bounded by (half of) the sum of
\begin{align*}
	c  \,  \sumx_{\beta (\mathrm{mod}\, c' n_1')}    \mathop{\sum_{\vnu (\mathrm{mod} \, n_1')} }_{(\beta+c'\vnu, n_1') = 1 }   \big|S_{\beta}^{+}  (\sqrt{q} t; c, n_1,d; N)\big|^2  
\end{align*}
and 
\begin{align*} 
	 c   \,  \sumx_{\beta (\mathrm{mod}\, c' n_1')}   \mathop{\sum_{\vnu (\mathrm{mod} \, n_1')} }_{(\beta+c'\vnu, n_1') = 1 }   \big|  {S_{\beta+c' \vnu}^{+} (\sqrt{q} t; c, n_1,d; N) } \big|^2 ,
\end{align*} 
while, by the change $ \beta + c' \vnu \ra \beta $, the second sum is turned into 
\begin{align*}
	c  \,  \sumx_{\beta (\mathrm{mod}\, c' n_1')}    \mathop{\sum_{\vnu (\mathrm{mod} \, n_1')} }_{(\beta-c'\vnu, n_1') = 1 }   \big|S_{\beta}^{+}  (\sqrt{q} t; c, n_1,d; N)\big|^2  . 
\end{align*}
By dropping the coprimality conditions $(\beta \pm c' \vnu, n_1') = 1$, we arrive at the bound: 
\begin{align*}
	c n_1 \sumx_{\beta (\mathrm{mod}\, c' n_1')}       \big|  {S_{\beta }^{+} (\sqrt{q} t; c, n_1,d; N) } \big|^2 . 
\end{align*} Recall here that $c' n_1' = cn_1/ d$.  So far, we have obtained the expression: 
\begin{align*}
	\frac{M T^{1+\vepsilon}} {N^2 n_1 \tau^{3/2} }    \sum_{q \Lt N/T} q^2   \int_{1/ \sqrt{2\tau}}^{1/\sqrt{\tau} }   \sum_{c \Lt N/ T q}   \frac 1 {c } \sum_{d\mid c n_1   }   d^2  \sumx_{\beta (\mathrm{mod}\, c n_1/d)}   \big| S^{+}_{\beta}  (\sqrt{q} t; c, n_1, d; N) \big|^2   \nd t. 
\end{align*}

\subsection{Large Sieve and Final Estimation} 
Note that we necessarily have $ d^2 q^3 \Lt N^2 n_1 \tau^3  $ as otherwise the $n$-sum would be empty.   Let us introduce the new variable $h = c n_1/d$ to simplify the sum above into
\begin{align*}
	 \frac{M T^{1+\vepsilon}} {N^2  \tau^{3/2} }   &  \mathop{\sum\sum}_{d^2 q^3 \Lt N^2 n_1 \tau^3} d q^2   \int_{1/ \sqrt{2\tau}}^{1/\sqrt{\tau} }   \sum_{h \Lt N n_1/ T d q}   \frac 1 { h }  \, \sumx_{\beta (\mathrm{mod}\, h)}      \\
 	& \cdot \bigg|  \sum_{n \shskip \asymp  {N^2 n_1 \tau^3} / {d^2 q^3}}      A(n , d)   e \bigg(   \frac {\widebar{\beta} n } {h}  \bigg) e \bigg(      \frac {2   \sqrt{n_1 qn} t} { h }  \bigg)  \varvv_{+} \bigg(\frac{n_1^2 n} { d h^3  }, \frac {n_1} { d h q   t^2} \bigg) \bigg|^2   \nd t   . 
\end{align*}
Finally,  the weight $\varvv_+$ is harmless as it may be handled by the Mellin inversion and the Cauchy--Schwarz inequality, at the loss of only $T^{\vepsilon}$,  and hence an application of  Lemma \ref{lem: Young's LS} with $\gamma =1/2$,  $\tau \ra 1/\sqrt{\tau}$, $v = 1/ 2 \sqrt{n_1 q}$,  $C = O (N n_1 /Tdq)$, and $N \ra O (N^2 n_1 \tau^3 / d^2 q^3)$   yields the estimate
\begin{align*}
	 \frac{M T^{1+\vepsilon}} {N^2  \tau^{3/2} }     \mathop{\sum\sum}_{d^2 q^3 \Lt N^2 n_1 \tau^3} d q^2 \bigg( \frac 1 {\sqrt{\tau}} \frac {N n_1} {T d q} + \frac 1 {\sqrt{n_1 q}} \frac { N \sqrt{n_1} \tau^{3/2} } {d q^{3/2} } \bigg) \sum_{n \shskip \asymp  {N^2 n_1 \tau^3} / {d^2 q^3}}     | A(n , d) |^2  & \\
	 \Lt \frac{M  } {N    }  T^{\vepsilon}     \sum_{q \Lt N^{2/3} n_1^{1/3} \tau } \bigg( \frac { n_1 q } { \tau^2} + T \bigg)  \mathop{\sum \sum}_{ d^2 n \Lt N^2 n_1 \tau^3 / q^3  } |A(n, d)|^2 & .
\end{align*}
By the Rankin--Selberg estimate \eqref{2eq: RS}, along with $\tau \Lt M^{\vepsilon} / M$,  this is   bounded by 
\begin{align*}
	\frac{M  } {N    }  T^{\vepsilon}      \sum_{q   } \bigg( \frac { n_1 q } {  \tau^2} + T  \bigg) \frac {N^2 n_1 \tau^3} {q^3} \Lt \big(M  n_1  {\tau}   +   {MT \tau^3}   \big) N n_1 T^{\vepsilon} \Lt \bigg( n_1 + \frac {T} {M^2  }  \bigg)   {N n_1 T^{ \vepsilon}}  , 
\end{align*}
as desired.




\begin{thebibliography}{AHLQ}
	
	\bibitem[AHLQ]{AHLQ-Bessel}
	K.~Aggarwal, R.~Holowinsky, Y.~Lin, and Z.~Qi.
	\newblock A {B}essel delta method and exponential sums for {${\rm GL}(2)$}.
	\newblock {\em Q. J. Math.}, 71(3):1143--1168, 2020.
	
	\bibitem[ALM]{Munshi-A-L-GL(3)}
	K.~Aggarwal, W.~H. Leung, and R.~Munshi.
	\newblock {S}hort second moment bound and subconvexity for {$\rm GL(3)$}
	{$L$}-functions.
	\newblock {\em J. Eur. Math. Soc. (JEMS)}, DOI 10.4171/JEMS/1667, 2025.
	
	\bibitem[BKY]{BKY-Mass}
	V.~Blomer, R.~Khan, and M.~P. Young.
	\newblock Distribution of mass of holomorphic cusp forms.
	\newblock {\em Duke Math. J.}, 162(14):2609--2644, 2013.
	
	\bibitem[Blo]{Blomer}
	V.~Blomer.
	\newblock Subconvexity for twisted {$L$}-functions on {${\rm GL}(3)$}.
	\newblock {\em Amer. J. Math.}, 134(5):1385--1421, 2012.
	
	\bibitem[CL]{Chandee-Li-GL(4)-Special-Points}
	V.~Chandee and X.~Li.
	\newblock The second moment of {$GL(4)\times GL(2)$} {$L$}-functions at special
	points.
	\newblock {\em Adv. Math.}, 365:107060, 39, 2020.
	
	\bibitem[DI]{DI-Nonvanishing}
	J.-M. Deshouillers and H.~Iwaniec.
	\newblock The nonvanishing of {R}ankin-{S}elberg zeta-functions at special
	points.
	\newblock {\em The {S}elberg {T}race {F}ormula and {R}elated {T}opics
		({B}runswick, {M}aine, 1984)}, {Contemp. Math.}, Vol. 53, 
	51--95. Amer. Math. Soc., Providence, RI, 1986.
	
	\bibitem[DIPS]{DIPS-Maass}
	J.-M. Deshouillers, H.~Iwaniec, R.~S. Phillips, and P.~Sarnak.
	\newblock Maass cusp forms.
	\newblock {\em Proc. Nat. Acad. Sci. U.S.A.}, 82(11):3533--3534, 1985.
	
	
	\bibitem[Gol]{Goldfeld}
	D.~Goldfeld.
	\newblock {\em Automorphic {F}orms and {$L$}-{F}unctions for the {G}roup {${\rm
				GL}(n, \text{\bf{R}})$}},  {Cambridge Studies in Advanced Mathematics, Vol. 99}.
	\newblock Cambridge University Press, Cambridge, 2006.
	
	\bibitem[GR]{G-R}
	I.~S. Gradshteyn and I.~M. Ryzhik.
	\newblock {\em Table of {I}ntegrals, {S}eries, and {P}roducts}.
	\newblock Elsevier/\hspace{0pt}Academic Press, Amsterdam, 7th edition, 2007.
	
	\bibitem[IK]{IK}
	H.~Iwaniec and E.~Kowalski.
	\newblock {\em Analytic {N}umber {T}heory},  {American
		Mathematical Society Colloquium Publications, Vol. 53}.
	\newblock American Mathematical Society, Providence, RI, 2004.
	
	\bibitem[IL]{Iwaniec-Li-Ortho}
	H.~Iwaniec and X.~Li.
	\newblock The orthogonality of {H}ecke eigenvalues.
	\newblock {\em Compos. Math.}, 143(3):541--565, 2007.
	
	\bibitem[Iwa1]{Iwaniec-L(1)}
	H.~Iwaniec.
	\newblock Small eigenvalues of {L}aplacian for {$\Gamma_0(N)$}.
	\newblock {\em Acta Arith.}, 56(1):65--82, 1990.
	
	\bibitem[Iwa2]{Iwaniec-Spectral-Weyl}
	H.~Iwaniec.
	\newblock The spectral growth of automorphic {$L$}-functions.
	\newblock {\em J. Reine Angew. Math.}, 428:\allowbreak139--159, 1992.
	
	\bibitem[Jut]{Jutila-LS}
	M.~Jutila.
	\newblock On spectral large sieve inequalities.
	\newblock {\em Funct. Approx. Comment. Math.}, 28:7--18, 2000.
	
	\delete{\bibitem[Kim]{Kim-Sarnak}
	H.~H. Kim.
	\newblock Functoriality for the exterior square of {${\rm GL}_4$} and the
	symmetric fourth of {${\rm GL}_2$}.
	\newblock {\em J. Amer. Math. Soc.}, 16(1):139--183, 2003.
	\newblock With appendix 1 by Dinakar Ramakrishnan and appendix 2 by Kim and
	Peter Sarnak.}
	
	\bibitem[KPY]{KPY-Stationary-Phase}
	E.~M. Kiral, I.~Petrow, and M.~P. Young.
	\newblock Oscillatory integrals with uniformity in parameters.
	\newblock {\em J. Th\'eor. Nombres Bordeaux}, 31(1):145--159, 2019.
	
	\bibitem[Kuz]{Kuznetsov}
	N.~V. Kuznetsov.
	\newblock {P}etersson's conjecture for cusp forms of weight zero and {L}innik's
	conjecture. {S}ums of {K}loosterman sums.
	\newblock {\em Math. Sbornik}, 39:299--342, 1981.
	
	\bibitem[Li]{XLi2011}
	X.~Li.
	\newblock Bounds for {${\rm GL}(3)\times {\rm GL}(2)$} {$L$}-functions and
	{${\rm GL}(3)$} {$L$}-functions.
	\newblock {\em Ann. of Math. (2)}, 173\allowbreak(1):301--336, 2011.
	
	\bibitem[Luo1]{Luo-Non-Vanishing}
	W.~Luo.
	\newblock On the nonvanishing of {R}ankin-{S}elberg {$L$}-functions.
	\newblock {\em Duke Math. J.}, 69(2):411--425, 1993.
	
	\bibitem[Luo2]{Luo-Twisted-LS}
	W.~Luo.
	\newblock The spectral mean value for linear forms in twisted coefficients of
	cusp forms.
	\newblock {\em Acta Arith.}, 70(4):377--391, 1995.
	
	\bibitem[Luo3]{Luo-LS}
	W.~Luo.
	\newblock Spectral mean-value of automorphic {$L$}-functions at special points.
	\newblock  {\em Analytic {N}umber {T}heory, {V}ol.\ 2 ({A}llerton {P}ark, {IL},
		1995)}, {Progr. Math.}, Vol. 139,  621--632. Birkh\"auser
	Boston, Boston, MA, 1996.
	
	\bibitem[Luo4]{Luo-Weyl}
	W.~Luo.
	\newblock Nonvanishing of {$L$}-values and the {W}eyl law.
	\newblock {\em Ann. of Math. (2)}, 154(2):477--502, 2001.
	
	\bibitem[Luo5]{Luo-2nd-Moment}
	W.~Luo.
	\newblock Refined asymptotics for second spectral moment of {R}ankin-{S}elberg
	{$L$}-functions at the special points.
	\newblock {\em Int. Math. Res. Not. IMRN}, (23):5457--5483, 2012.
	
	\bibitem[Mon]{Montgomery-Topics}
	H.~L. Montgomery.
	\newblock {\em Topics in {M}ultiplicative {N}umber {T}heory},    {Lecture Notes in Mathematics, Vol. 227}.
	\newblock Springer-Verlag, Berlin-New York, 1971.
	
	\bibitem[MS1]{Miller-Schmid-2006}
	S.~D. Miller and W.~Schmid.
	\newblock Automorphic distributions, {$L$}-functions, and {V}oronoi summation
	for {${\rm GL}(3)$}.
	\newblock {\em Ann. of Math. (2)}, 164(2):423--488, 2006.
	
	\bibitem[MS2]{Miller-Schmid-2009}
	S.~D. Miller and W.~Schmid.
	\newblock A general {V}oronoi summation formula for {${\rm GL}(n,\mathbb{Z})$}.
	\newblock   {\em Geometry and {A}nalysis. {N}o. 2}, Adv. Lect. Math., Vol. 18,
	  173--224. Int. Press, Somerville, MA, 2011.
	
	\bibitem[MV]{Michel-Venkatesh-GL2}
	P.~Michel and A.~Venkatesh.
	\newblock The subconvexity problem for {${\rm GL}_2$}.
	\newblock {\em Publ. Math. Inst. Hautes \'Etudes Sci.}, (111):171--271, 2010.
	
	\bibitem[MZ]{MZ-Voronoi}
	S.~D. Miller and F.~Zhou.
	\newblock The balanced {V}oronoi formulas for {${\rm GL}(n)$}.
	\newblock {\em Int. Math. Res. Not. IMRN}, (11):3473--3484, 2019.
	
	\bibitem[PS]{Phillips-Sarnak}
	R.~S. Phillips and P.~Sarnak.
	\newblock On cusp forms for co-finite subgroups of {${\rm PSL}(2,{\bf R})$}.
	\newblock {\em Invent. Math.}, 80(2):339--364, 1985.
	
	\bibitem[Qi1]{Qi-Bessel}
	Z.~Qi.
	\newblock Theory of fundamental {B}essel functions of high rank.
	\newblock {\em Mem. Amer. Math. Soc.}, 267\allowbreak(1303):vii+123, 2020.
	
	\bibitem[Qi2]{Qi-GL(3)}
	Z.~Qi.
	\newblock Subconvexity for {$L$}-functions on {$\rm GL_3$} over number fields.
	\newblock {\em J. Eur. Math. Soc. (JEMS)}, 26\allowbreak(3):1113--1192, 2024.
	
	\bibitem[Sog]{Sogge}
	C.~D. Sogge.
	\newblock {\em Fourier {I}ntegrals in {C}lassical {A}nalysis},  
	{Cambridge Tracts in Mathematics, Vol. 105}.
	\newblock Cambridge University Press, Cambridge, 1993.
	
	\bibitem[Tit]{Titchmarsh-Riemann}
	E.~C. Titchmarsh.
	\newblock {\em The {T}heory of the {R}iemann {Z}eta-{F}unction}.
	\newblock The Clarendon Press, Oxford University Press, New York, 2nd
	edition, 1986. 
	
	\bibitem[Wat]{Watson}
	G.~N. Watson.
	\newblock {\em A {T}reatise on the {T}heory of {B}essel {F}unctions}.
	\newblock Cambridge University Press, Cambridge, England; The Macmillan
	Company, New York, 1944.
	
	\bibitem[You]{Young-GL(3)-Special-Points}
	M.~P. Young.
	\newblock The second moment of {$GL(3)\times GL(2)$} {$L$}-functions at special
	points.
	\newblock {\em Math. Ann.}, 356(3):1005--1028, 2013.
	
\end{thebibliography}
\end{document}